\documentclass[reqno, 11pt]{amsart}
\usepackage{mathtools}
\usepackage{amsmath}
\usepackage{amssymb}
\usepackage{yhmath}
\usepackage{graphicx}
\usepackage{ mathrsfs }
\usepackage{bbm}
\usepackage{xcolor}
\usepackage{tikz-cd}
\usepackage{tikz}
\usetikzlibrary{patterns}

\DeclareMathAlphabet{\mathpzc}{OT1}{pzc}{m}{it}

\usepackage{thmtools}
\usepackage{thm-restate}

\usepackage{caption}

\newtheorem{thm}{Theorem}[section]
\newtheorem{lem}[thm]{Lemma}
\newtheorem{prop}[thm]{Proposition}
\newtheorem*{prop*}{Proposition}
\newtheorem{cor}[thm]{Corollary}
\newtheorem*{cor*}{Corollary}
\newtheorem*{lem*}{Lemma}

\numberwithin{thm}{section}

\theoremstyle{definition}

\newtheorem{Def}[thm]{Definition}

\newtheorem{Rmk}[thm]{Remark}
\newtheorem{ex}[thm]{Example}

\numberwithin{equation}{section}

\def\CC{{\mathbb C}}

\def\NN{{\mathbb N}}

\def\RR{{\mathbb R}}

\def\scrC{{\mathcal C}}

\def\scrF{{\mathcal F}}
\def\scrH{{\mathcal H}}

\def\scrS{{\mathcal S}}

\def\scrU{{\mathcal U}}

\def\e{\mathrm{e}}
\def\i{\mathrm{i}}

\def\dim{\operatorname{dim}}

\def\G{\operatorname{G{}}}
\def\L{\operatorname{L{}}}

\def\SO{\operatorname{SO}}

\def\PSL{\operatorname{PSL}}

\def\supp{\operatorname{supp}}

\def\vol{\operatorname{vol}}

\def\Onder#1#2#3#4#5{#1 \setbox0=\hbox{$#1$}\setbox1=\hbox{$#2$}
       \dimen0=.5\wd0 \dimen1=\dimen0 \dimen2=\dp0 \dimen3=\dimen2
       \advance\dimen0 by .5\wd1 \advance\dimen0 by -#4
       \advance\dimen1 by -.5\wd1 \advance\dimen1 by -#4
       \advance\dimen2 by -#3 \advance\dimen2 by \ht1
       \advance\dimen2 by 0.3ex \advance\dimen3 by #5
        \kern-\dimen0\raisebox{-\dimen2}[0ex][\dimen3]{\box1}
       \kern\dimen1}

\newcommand{\GaG}{\Gamma\backslash G}

\newcommand{\sfrac}[2]{{\textstyle \frac {#1}{#2}}}

\newcommand{\F}{\mathcal{F}}

\newcommand{\fg}{\mathfrak{g}}

\newcommand{\fa}{\mathfrak{a}}
\newcommand{\fb}{\mathfrak{b}}
\newcommand{\fh}{\mathfrak{h}}

\newcommand{\fm}{\mathfrak{m}}
\newcommand{\fn}{\mathfrak{n}}

\newcommand{\Ug}{\mathcal{U}(\mathfrak{g}_{\mathbb{C}})}
\newcommand{\Uh}{\mathcal{U}(\mathfrak{h}_{\mathbb{C}})}

\newcommand{\Zg}{\mathcal{Z}(\mathfrak{g}_{\mathbb{C}})}

\newcommand{\Ad}{\mathrm{Ad}}

\newcommand{\rank}{\mathrm{rank}}

\newcommand{\ba}{\backslash}
\newcommand{\bH}{\mathbb H}
\newcommand{\op}{\operatorname}

\newcommand{\la}{\lambda}
\renewcommand{\epsilon}{\varepsilon}
\newcommand{\BR}{\op{BR}}

\newcommand{\br}{\mathbb R}
\renewcommand{\G}{\Gamma}
\newcommand{\Ga}{\Gamma}
\newcommand{\cal}{\mathcal}
\newcommand{\La}{\Lambda}
\newcommand{\be}{\begin{equation}}
\newcommand{\ee}{\end{equation}}
\newcommand{\ga}{\gamma}
\newcommand{\inte}{\op{int}}

\newcommand{\dg}{D_\Gamma^\star}
\renewcommand{\L}{\mathcal L}

\renewcommand{\c}{\mathbb C}
\renewcommand{\la}{\langle}
\newcommand{\ra}{\rangle}

\renewcommand{\e}{\epsilon}
\begin{document}
\title[Positive eigenfunctions]{Temperedness of $L^2(\Gamma\ba G)$ and positive eigenfunctions in higher rank.}

\author{Sam Edwards and Hee Oh}
\address{Department of Mathematical Sciences, Durham University, Lower Mountjoy, DH1 3LE Durham, United Kingdom}

\address{Department of Mathematics, Yale University, New Haven, CT 06511 and Korea Institute for Advanced Study, Seoul}

\email{samuel.c.edwards@durham.ac.uk}

\email{hee.oh@yale.edu}
\thanks{Edwards was supported by funding from the Heilbronn Institute for Mathematical Research and Oh was supported in part by NSF grant DMS-1900101}

\begin{abstract} Let $G=\op{SO}^\circ(n,1) \times \op{SO}^\circ(n,1)$ and $X=\bH^{n}\times \bH^{n}$ for $n\ge 2$. For a pair $(\pi_1, \pi_2)$ of non-elementary convex cocompact representations of a  finitely generated group $\Sigma$ into $\op{SO}^\circ(n,1)$, let $\G=(\pi_1\times \pi_2)(\Sigma)$. Denoting the bottom of the $L^2$-spectrum of the negative Laplacian on $\Ga\ba X$ by $\lambda_0$, we show:
\begin{enumerate}
    \item $L^2(\Ga\ba G)$ is tempered and  $\lambda_0=\frac{1}{2}(n-1)^2$;
    \item There exists no positive Laplace eigenfunction in $L^2(\Gamma \ba X)$.
\end{enumerate}
\medskip 
In fact, analogues of (1)-(2) hold for any Anosov subgroup $\Ga$ in the product of at least two simple algebraic groups of rank one as well as
 for Hitchin subgroups $\G<\PSL_d(\br)$, $d\ge 3$. 
 Moreover, if $G$ is a semisimple real algebraic group of rank at least $2$, then (2) holds for any Anosov subgroup $\Ga$ of $G$.  
\end{abstract}





\maketitle
\tableofcontents

\section{Introduction}

\subsection*{Motivation and background} 
 Locally symmetric spaces provide key examples of Riemannian manifolds for which there exist numerous tools for studying various aspects of spectral geometry. For example, properties of dynamical systems related to the manifold are closely connected to the spectral theory of the Laplace operator, as well as to representation theory. While the spectral theory of finite-volume locally symmetric spaces has been quite extensively developed, the infinite volume setting provides many examples of interesting phenomena that are less well understood. Nevertheless, for rank one locally symmetric spaces of infinite volume, a number of key facts about the spectrum have been established. 

Let $(\bH^n,d)$, $n\ge 2$, denote the $n$-dimensional hyperbolic space of constant curvature $-1$, and let $G=\op{Isom}^+(\bH^n)\simeq \op{SO}^\circ (n,1)$ denote the group of all orientation preserving isometries of $\bH^n$. Let $\Ga<G$ be a torsion-free\footnote{all discrete subgroups in this paper will be assumed to be torsion-free} discrete subgroup.
The critical exponent $0\le \delta=\delta_\Ga\le n-1$ is defined as the abscissa of convergence of the Poincar\'e series $\sum_{\gamma\in \G} e^{-s d(o, \ga o)}$ for $o\in \bH^n$. 
We denote by $\Delta$ the hyperbolic Laplacian and by $\lambda_0=\lambda_0(\Ga\ba \bH^n)$
the bottom of the $L^2$-spectrum of the negative Laplace operator $-\Delta$, which is given as  
 \be\label{la33} \lambda_0:= \inf\left\lbrace \frac{\int_{\Gamma\ba \bH^n}\|\text{grad} \, f \|^2\,d\vol}{\int_{\Gamma\ba \bH^n}|f|^2\,d\vol}\,:\,f\in C^\infty_c(\Gamma\ba \bH^n)\right\rbrace\ee 
(see \cite[Theorem 2.2]{Su}).
In a series of papers, Elstrodt (\cite{E1}, \cite{E2}, \cite{E3}) and Patterson (\cite{P1}, \cite{P2}, \cite{P3}) developed the relationship between 
$\delta$ and $\lambda_0$, proving the following theorem for $n=2$. The general case is due to Sullivan
\cite[Theorem 2.21]{Su}.

\begin{thm}[Generalized Elstrodt-Patterson I] \label{ep} For any discrete subgroup $\Ga<\op{SO}^\circ(n,1)$,
the following are equivalent:
\begin{enumerate}
  \item $\delta\le \frac{1}{2}(n-1)$;
    \item $\lambda_0=\frac{1}{4}(n-1)^2$.
 \end{enumerate}
\end{thm}

The right translation action of $G$ on the quotient space $\Gamma\ba G$ equipped with a $G$-invariant measure gives rise to a unitary representation of $G$ on the Hilbert space $L^2(\Gamma\ba G)$, called a quasi-regular representation of $G$. If we set $K\simeq \SO(n)$ to be a maximal compact subgroup of $G$ and identify $\bH^n$ with $G/K$,
the space  of $K$-invariant functions of $L^2(\Gamma\ba G)$ can be identified with
$L^2(\Gamma\ba \bH^n)$. The bottom of the $L^2$-spectrum $\lambda_0$ then provides  information on which complementary series representation of $G$ can occur in $L^2(\Gamma\ba G)$. Indeed, it follows from the classification of the
unitary dual of $\op{SO}^\circ(n,1)$ that $\lambda_0=(n-1)^2/4$
is equivalent to saying that the quasi-regular representation $L^2(\Gamma\ba G)$ does not contain any complementary series representation (cf. \cite{Su}, \cite{EO}), which is again equivalent to the {\it temperedness} of $L^2(\Gamma\ba G)$. As first introduced by Harish-Chandra \cite{Ha},  a unitary representation $(\pi, \cal H_{\pi})$ of a semisimple real algebraic group $G$ is tempered (Definition \ref{te}) 
if all of its matrix coefficients belong to
$L^{2+\epsilon}(G)$ for any $\e>0$, or, equivalently, if
$\pi$ is weakly contained\footnote{$\pi$ is weakly contained in a unitary representation $\sigma$ of $G$ if any diagonal
matrix coefficients of $\pi$ can be approximated, uniformly on compact sets, by convex
combinations of diagonal matrix coefficients of $\sigma$.}
in the regular representation $L^2(G)$ (\cite{CHH}, see Proposition \ref{chh}).

Therefore Theorem \ref{ep} can be rephrased as follows:
\begin{thm}[Generalized Elstrodt-Patterson II]\label{ep2}  For any discrete subgroup $\Ga<G=\op{SO}^\circ(n,1)$,
the following are equivalent:
\begin{enumerate}
    \item $\delta\le \frac{1}{2}(n-1)$;
\item   $L^2(\Ga\ba G)$ is  tempered.
 \end{enumerate}
\end{thm}

The size of the critical exponent $\delta$ is also related to the existence of a square-integrable positive Laplace eigenfunction on $\Ga\ba \bH^n$. A discrete subgroup $\G<G$ is called convex cocompact if there exists a convex subspace of $\bH^n$ on which $\G$ acts co-compactly. For convex cocompact subgroups of $G$ (more generally for geometrically finite subgroups), Patterson and Sullivan showed
the following using their theory of conformal measures on the boundary $\partial\bH^n$ (\cite{Pa}, \cite{Su2}, \cite[Theorem 2.21]{Su}):
\begin{thm}[Sullivan]\label{one} For a convex cocompact subgroup  $\Ga<\op{SO}^\circ(n,1)$, the following are equivalent:
\begin{enumerate}
    \item $\delta\le \frac{1}{2}(n-1)$;
 \item There exists no positive Laplace eigenfunction in $L^2(\Gamma\ba \bH^n)$.
 \end{enumerate}
\end{thm}
Since $\lambda_0$ divides the positive spectrum and the $L^2$-spectrum on $\Ga\ba \bH^n$ by Sullivan's theorem \cite[Theorem 2.1]{Su} (see Theorem \ref{no}), (2) is equivalent to saying that any $\lambda_0$-harmonic function
(i.e., $-\Delta f=\lambda_0 f$) on $\Ga\ba \bH^n$ is not square-integrable.

\subsection*{Main results} 
The main aim of this article is to discuss analogues of Theorems \ref{ep}, \ref{ep2}, and \ref{one} for a certain class of discrete subgroups of a connected semisimple real algebraic group of higher rank, i.e., rank at least $2$.

We begin by describing a special case of our main theorem when
$G=\SO^\circ (n_1,1) \times \SO^\circ (n_2,1)$ with $n_1, n_2\ge 2$.
Let $X$ be the Riemannian product $\bH^{n_1}\times \bH^{n_2} $ and $\Delta$ the Laplace-Beltrami operator on $X$.
For a torsion-free discrete subgroup $\Ga<G$,
a smooth function $f$ on $\Ga\ba X$ is called $\lambda$-harmonic if  $-\Delta f=\lambda f$. The number $\lambda_0=\lambda_0(\Ga\ba X)$ is given in the same way as \eqref{la33} replacing $\Gamma\ba \bH^n$ by $\Gamma\ba X$.

\begin{thm} \label{m} 
Let \be\label{as} \G=(\pi_1\times \pi_2)(\Sigma)=\{(\pi_1(\sigma), \pi_2(\sigma))\in G:\sigma\in \Sigma\}\ee 
where $\pi_i:\Sigma\to \SO^\circ (n_i,1)$ is a non-elementary
convex cocompact representation of a finitely generated group $\Sigma$ for $i=1,2$. Then 
\begin{enumerate}
    \item $L^2(\Gamma\ba G)$ is tempered and $\lambda_0 =\frac{1}{4} ((n_1-1)^2 +(n_2-1)^2)$;
    \item There exists no positive Laplace eigenfunction in $L^2(\Gamma \ba X)$, or equivalently, no $\lambda_0$-harmonic function is  square-integrable.
\end{enumerate}
\end{thm}

\begin{Rmk} 
Theorem \ref{m} does not hold for a general subgroup $\Ga<G$ of infinite co-volume. For example, if $\Ga<\op{SO}^\circ(n_1,1)\times \op{SO}^\circ(n_2,1)$ is the product of two convex cocompact subgroups, each of which having critical exponent greater than $\frac{1}{2}{(n_i-1)}$, then $L^2(\Gamma\ba G)$ is not tempered and $L^2(\Gamma\ba X)$ possesses a positive Laplace eigenfunction.
\end{Rmk}
\medskip

We now discuss a general setting. Let $G$ be a connected semisimple real algebraic group and $X$ the associated Riemannian symmetric space.
In the rest of the introduction, we assume that $\Gamma<G$ is a  torsion-free Zariski dense discrete subgroup. We let $\psi_\Ga:\fa \to \br \cup\{-\infty\}$ denote the \emph{growth indicator function} of $\Ga$ as defined in \eqref{grow3}, where $\fa$ is the Lie algebra of a maximal real split torus of $G$.
The function $\psi_\Ga$ can be regarded as a higher rank generalization of the critical exponent of $\Ga$. Let $\rho$ denote the half sum of all positive roots for $(\mathfrak g, \fa)$, counted with multiplicity. 
Analogous to the fact that the critical exponent
$\delta$ is always bounded above by $ n-1$ for a discrete subgroup $\Gamma<\SO^\circ(n,1)$, we have
the upper bound $\psi_\Ga\le 2\rho$ for any discrete subgroup $\Gamma$ of $G$ \cite{Quint2}.

 The following
 Theorem \ref{m33} generalizes Theorems \ref{ep}, \ref{ep2}, and 
 \ref{one} to Anosov subgroups of $G$ (with respect to a minimal parabolic subgroup of $G$) which are regarded as higher rank generalizations of convex cocompact subgroups.
  For $G=\SO^\circ(n_1,1)\times \SO^\circ(n_2,1)$, they are precisely given by the class of subgroups considered  in Theorem \ref{m}. We refer to Definition \ref{Ano} for a general case.
We mention that they were first introduced by Labourie \cite{La} for surface groups and then generalized by Guichard and Wienhard  for hyperbolic groups \cite{GW} (see also \cite{GG}, \cite{KLP}).

In the following theorem,
 the norm $\|\rho\|$ is defined via the identification $\fa^*$ and $\fa$
using the Killing form on $\mathfrak g$. 
Denote by $\sigma(\Ga\ba X)$ the $L^2$-spectrum of $-\Delta$ on $\Ga\ba X$.
\begin{thm} \label{m33} Let $G$ be a connected semisimple real algebraic group and
$\G$ a Zariski dense Anosov subgroup of $G$. 
The following (1)-(3)  are equivalent, and imply (4):
\begin{enumerate}
    \item  $\psi_\Ga \le \rho$;
 \item  $L^2(\Gamma\ba G)$ is tempered and $\lambda_0(\Ga\ba X) =\| \rho\|^2$;
 \item  $L^2(G)$ and $L^2(\Ga\ba G)$  are weakly contained in each other and $\sigma(\Ga\ba X)=\sigma(X)=[\|\rho\|^2, \infty)$;
 \item There exists no positive Laplace eigenfunction in $L^2(\Gamma\ba X)$.
\end{enumerate}
 Moreover, if $\op{rank}G\ge 2$, then (4) always holds
for any Anosov subgroup $\Ga<G$.
\end{thm}

Our proof of the implication $(1)\Rightarrow (2)$ is based on the asymptotic behavior of the Haar matrix coefficients for Anosov subgroups obtained in \cite{ELO} and \cite{CS} as well as Harish-Chandra's Plancherel formula (see Theorems \ref{tlam} and \ref{tem}). The implication $(2)\Rightarrow (1)$ is true for a general discrete subgroup (see the proof of Theorem \ref{tem}). 
The equivalence $(2)\Leftrightarrow (3)$ uses the observation that
$L^2(G)$ is weakly contained in $L^2(\Ga\ba G)$ whenever the injectivity radius of $\Ga\ba G$ is infinite, and that
$\Ga\ba G$ has infinite injectivity radius for any Anosov subgroup $\Ga<G$, except for cocompact lattices of a rank one Lie group (see Section \ref{injs}).
For (4), we first prove that any positive Laplace eigenfunction in $L^2(\Ga\ba X)$
is indeed a joint eigenfunction  for the whole ring of $G$-invariant differential operators, which then can be studied via $\Gamma$-conformal measures on the Furstenberg boundary of $G$ (see Sections \ref{sje} and \ref{sje2}). We establish a higher rank version of Sullivan-Thurston's smearing theorem (Theorem \ref{smear}) from which we deduce the non-existence of square-integrable positive Laplace eigenfunctions for any higher rank Anosov subgroup (see Section 7 and Corollary \ref{ttt}).
 When $\text{rank }G=1$, Anosov subgroups are convex cocompact groups and the implication $(1)+(2)\Rightarrow (4)$ is obtained in \cite{Su} (see also \cite[Theorem 3.1]{RT}) for $X=\bH^n$ and in \cite{WW} in general.

Although the condition $\psi_\Ga \le \rho$ may appear quite strong, it was verified in a recent work of Kim-Minsky-Oh \cite{KMO}
for Anosov subgroups
in the following setting, and hence we deduce from Theorem \ref{m33}:
\begin{thm} \label{oo} 
Let $\G$ be a Zariski dense Anosov subgroup of the product of at least two simple real algebraic groups of rank one, or a Zariski dense Anosov subgroup of a Hitchin subgroup of $\PSL_d(\br)$ for $d\ge 3$. Then (1)-(4) of Theorem \ref{m33} hold. 
\end{thm}

It is conjectured in \cite{KMO} that
any Anosov subgroup of a higher rank semisimple real algebraic group satisfies the condition $\psi_\Ga\le \rho$. This conjecture suggests that Anosov subgroups in higher rank groups are more like generalizations of
convex cocompact subgroups of {\it small} critical exponent.

\medskip 

\noindent{\bf Groups of the second kind and positive joint eigenfunctions.}  For any discrete subgroup $\G$ which is not cocompact in $G$ and for any $\lambda \le \lambda_0(\Ga\ba X)$,
Sullivan proved the existence of
a positive $\lambda$-harmonic function. 
We prove a higher-rank strengthening of this result: for any discrete subgroup
of the \emph{second kind} (see Definition \ref{gp}) whose limit cone is contained in the interior of $\fa^+$  and  for any linear form $\psi\ge \psi_\Gamma$, we construct a positive joint eigenfunction with character corresponding to $\psi$ (Theorem \ref{exi}).

\medskip 

\noindent{\bf Organization:}  In section 2, we
review the basic notions and notations which will be used throughout the paper. In section 3,
we show that any postive joint eigenfunction on $\Gamma\ba X$ (i.e., an eigenfunction for the whole ring of $G$-invariant differential operators) arises from a $(\Gamma,\psi)$-conformal density (Proposition \ref{d2}). In section 4, we compute the Laplace eigenvalue of a positive joint eigenfunction associated to a $(\Ga, \psi)$-conformal measure
(Proposition \ref{LaplaceEV}). In section 5, we introduce the notion of subgroups of the second kind. We then construct positive joint eigenfunctions 
for any $\psi\ge \psi_\Ga$ for any subgroup of the second kind with its limit cone contained in $\inte \fa^+\cup\{0\}$ (Theorem \ref{exi}). In section 6, we compute the $L^2$-spectrum of $X$ (Theorem \ref{ss}) and show
that $\lambda_0=\|\rho\|^2$ if $L^2(\Ga\ba G)$ is tempered (Theorem \ref{tlam}). We  
show that a positive  Laplace eigenfunction in $L^2(\Ga\ba X)$
is necessarily a joint eigenfunction (Corollary \ref{lzero})
and a spherical vector of a unique irreducible subrepresentation of $L^2(\Ga\ba G)$ (Theorem \ref{rep}).
In section 7, we prove a higher rank version of Sullivan-Thurston's smearing theorem (Theorem \ref{smear}) to
obtain the non-existence theorem of $L^2$-positive Laplace eigenfunctions in higher rank. In section 8, we prove the weak containment $L^2(G)\propto L^2(\Ga\ba G)$ for all Anosov subgroups $\G$ in higher rank groups. In section 9, we prove the equivalence of the temperedness of $L^2(\Ga\ba G)$ and $\psi_\Ga\le\rho$ (Theorem \ref{tem}). We also
 deduce Theorem \ref{m33}.

\medskip

 \noindent{\bf Acknowledgements:} We would like to thank Marc Burger for bringing the reference \cite{Su} to our attention.
 We would also like to thank Dick Canary, Francois Labourie,  Curt McMullen and Dennis Sullivan for useful conversations.

\section{Preliminaries and notations}
Let $G$ be a connected semisimple real algebraic group, i.e., the identity component of the group of real points of a semisimple algebraic group defined over $\br$. Let $\Ga<G$ be a torsion-free discrete subgroup.
Let $P$ be a minimal parabolic subgroup of $G$ with a fixed Langlands decomposition $P=MAN$ where $A$ is a maximal real split torus of $G$, $M$ is the maximal compact subgroup of $P$, which commutes with $A$, and $N$ is the unipotent radical of $P$. We denote by $\fg, \fa, \fn$ respectively the Lie algebras of $G, A, N$.
We fix a positive Weyl chamber $\fa^+\subset \fa$ so that $\fn$ consists of positive root subspaces.
Let $\Sigma^+$ denote the set of all  positive roots for $(\fg, \fa^+)$. We also write $\Pi\subset \Sigma^+$ for the set of all simple roots.
We denote by $$\rho=\frac{1}{2}\sum_{\alpha\in \Sigma^+}\alpha$$ the half sum of the positive roots for $(\mathfrak g, \mathfrak a^+)$, counted with multiplicity.
We denote by $\langle\cdot,\cdot\rangle$ and $\|\cdot\|$ the inner product and norm on $\fg$ respectively, induced from the Killing form: $B(x,y)=\op{Tr} (\op{ad}x \op{ad}(y))$ for $x, y\in \fg$.

We fix a maximal compact subgroup $K$ of $G$ so that the
Cartan decomposition $G=K (\exp \fa^+) K$ holds, that is, for any $g\in G$, there exists a unique element $\mu(g)\in \fa^+$ such that $g\in K \exp \mu(g) K$. We call the map $\mu:G\to \fa^+$ the Cartan projection map. 

The Riemannian symmetric space $(X,d)$ can be identified with the quotient space $G/K$ with the metric $d$ induced from $\langle\cdot,\cdot\rangle$. We denote by $d\op{vol}$ the Riemannian volume form on $X$ or on $\Gamma\ba X$. We also use $dx$ to denote this volume form as well as the Haar measure on $G$, or on $\Ga\ba G$.
In particular, $d(\cdot,\cdot)$ will denote both the left $G$-invariant Riemannian distance function on $X$, as well as the left $G$-invariant and right $K$-invariant distance on $G$. We set $o=[K]\in X$. We then have
$\|\mu(g)\|=d(go, o)$ for $g\in G$. We do not distinguish a function on $X$ and a right $K$-invariant function on $G$.

Let $w_0\in K$ be an element of the normalizer of $A$ 
so that $\op{Ad}_{w_0}\mathfrak a^+= -\mathfrak a^+$.
The opposition involution  $\i:\mathfrak a \to \mathfrak a$ is defined by 
\be\label{opp} \i (u)= -\op{Ad}_{w_0} (u)\quad\text{ for all $u\in \fa$}.\ee 

Let $\cal F:=G/P$ denote the Furstenberg boundary of $G$. 
We define the following {\it visual} maps $G\to \cal F$: for each $g\in G$, 
   \be\label{visual} g^+:=gP\in \cal F\quad\text{and}\quad g^-:=gw_0P\in \cal F.\ee 

 The unique open $G$-orbit $\F^{(2)}$ in $\F\times \F$ under the diagonal $G$-action is given by:
$$\F^{(2)}=G(e^+, e^-)=\{(g^+, g^-)\in \F\times \F: g\in G\}.$$ 
Two points $\xi,\eta$ in $\F$ are said
to be in general position if $(\xi, \eta)\in\F^{(2)}$.

\medskip 

\noindent{\bf Conformal measures.}
Let $G=KAN$ be the Iwasawa decomposition, $\kappa:G\rightarrow K$ the $K$-factor projection of this decomposition, and $H:G\rightarrow \fa$ be the Iwasawa cocycle defined by the relation: for $g\in G$,
$$ g\in \kappa(g)\exp\big(H(g)\big)N.$$
Note that $K$ acts transitively on $\F$ and $K\cap P=M$, and hence we may identify $\cal F$ with $K/M$.
The Iwasawa decomposition can be used to describe both the action of $G$ on $\scrF=K/M$ and the $\fa$-valued Busemann map as follows: for all $g\in G$ and $[k]\in\scrF$ with $k\in K$,
$$ g\cdot [k]=[\kappa(gk)],$$
and the $\fa$-valued Busemann map is defined by
$$ \beta_{[k] }(g(o),h(o)):=H(g^{-1}k)-H(h^{-1}k)\in \fa \quad\text{
for all $g,h\in G$.}$$

We denote by $\fa^*=\text{Hom}_\br(\fa, \br)$ the space of all linear forms on $\fa$.
\begin{Def} \label{conf_def} Let $\psi\in\fa^*$.
\begin{enumerate}
    \item A  finite Borel
measure $\nu_o$ on $\scrF=K/M$ is said to be a $(\Ga,\psi)$-conformal measure (with respect to $o=[K]$) if for all $\ga\in \Ga$ and $\xi=[k]\in K/M$,
$$ \frac{d\ga_{\ast}\nu_o}{d\nu_o}(\xi)=e^{-\psi(\beta_\xi (\ga o,  o))}=
e^{-\psi\big(H(\gamma^{-1}k)\big)},$$
or equivalently
$$d\nu_o([k])=e^{\psi\big(H(\ga k)\big)}d\nu_o(\ga \cdot[k]), $$
where $\gamma_*\nu_o(Q)=\nu_o(\gamma^{-1}Q)$ for
any Borel subset $Q\subset \cal F$. Unless mentioned otherwise, all conformal measures in this paper are assumed to be with respect to $o$.
\item A collection $\{\nu_x: x\in X\} $ of finite Borel measures on $\F$
is called a $(\Ga, \psi)$-conformal density if 
for all $x, y\in X$, $\xi\in \F$ and $\ga\in \Ga$,
\be \label{c0} \frac{d\nu_x}{d\nu_y}(\xi)= e^{-\psi (\beta_\xi (x,y))}\quad\text{and}\quad d\gamma_* \nu_x =d\nu_{\ga (x)}.\ee 
\end{enumerate}

\end{Def}

A $(\Gamma,\psi)$-conformal  measure $\nu_o$ defines a $(\Ga, \psi)$-conformal density $\{\nu_x: x\in X\}$ by the formula:
$$d\nu_x(\xi) =e^{-\psi(\beta_\xi (x, o))} d\nu_o(\xi) ,$$
and conversely any $(\Ga,\psi)$-conformal density $\{\nu_x\}$
is uniquely determined by its member $\nu_o$ by \eqref{c0}.
For this reason, by abuse of terminology,  we sometimes do not distinguish conformal measures and conformal densities.

\medskip 
\noindent{\bf Growth indicator function.} Let $\Ga<G$ be a Zariski dense discrete subgroup. Following Quint \cite{Quint2}, let $\psi_\Ga:\fa \to \br \cup\{-\infty\}$ denote the growth indicator function of $\Ga$: 
for any non-zero $v\in \fa$, \begin{equation}\label{grow3}\psi_{\Gamma}(v):=\|v\| 
\inf_{v\in\cal C} \tau_{\cal C},\end{equation}
where the infimum is over all open cones $\cal C$ containing $v$
and $\tau_{\cal C}$ denotes the abscissa of convergence of the series $\sum_{\ga\in\Ga,\,\mu(\ga)\in\cal C}e^{-s\|\mu(\ga)\|}$. For $v=0$, we let $\psi_\Gamma(0)=0$.
We note that $\psi_\Ga$ does not change if we 
replace the norm $\|\cdot\|$ by any other
norm on $\fa$.
For any discrete subgroup $\Ga<G$, we have the upper bound $\psi_\Ga\le 2\rho$ \cite{Quint2}.
On the other hand, 
when $\Gamma$ is of infinite co-volume in a simple real algebraic group of rank at least $2$, 
Quint deduced from \cite{Oh} that $\psi_\Gamma \le 2\rho-\eta_G $,
where $\eta_G$ is the half sum of a maximal strongly orthogonal subset of the root system of $G$  (\cite{Quint3}, see also \cite[Theorem 7.1]{LO2}).

\medskip
\noindent{\bf Limit cone and limit set.} 
The limit cone $\cal L=\cal L_\Ga$ of $\Ga$ is defined as
  the asymptotic cone of $\mu(\Ga)$, i.e.,
 $$\L=\{\lim t_i\mu(\ga_i)\in \fa^+: t_i\to 0, \gamma_i\in \Ga\}.$$ Benoist showed that
for $\Ga$ Zariski dense, $\cal L$ is a convex cone with non-empty interior \cite{Ben}.
 Quint \cite{Quint2} showed that $\psi_\Ga$ is a concave and upper-semicontinuous function such that
$\psi_\Ga\ge 0$ on $\cal L$, $\psi_\Ga>0 $ on $\inte \L$ and $\psi_{\Ga}=-\infty$ outside $\L$.

For a sequence $g_i\in  G$, we write $g_i\to \infty$ regularly
if $\alpha(\mu (g_i))\to \infty$ for all $\alpha\in \Pi$.
For $g\in G$, we write 
$g=\kappa_1(g)\exp (\mu(g))\kappa_2(g)\in KA^+K$; if $\mu(g)\in \inte \fa^+$, then $[\kappa_1(g)]\in K/M=\cal F$ is well-defined.
\begin{Def}\label{conver}
 A sequence $p_i \in X$ is said to converge to $\xi\in \cal F$ and we write $\lim_{i\to \infty} p_i=\xi$ if 
 there exists a sequence $g_i\to \infty$ regularly in $G$ with $p_i=g_i(o)$
 and $\lim_{i\to\infty}[\kappa_1(g_i)]=\xi$.
\end{Def}

We denote by $\La\subset \F$ the limit set of $\G$, which is defined as
\be\label{ls} \La=\{\lim \ga_i(o)\in \cal F: \ga_i\in \Ga\}.\ee 
For $\Ga<G$ Zariski dense, this is the unique $\Ga$-minimal subset of $\F$ (\cite{Ben}, \cite{LO}).
\medskip

\noindent{\bf Tangent linear forms.}
We set 
\be\label{dga} D_\Ga=\{ \psi\in \fa^*: \psi \ge \psi_\Ga\}.\ee 
A linear form $\psi\in \fa^*$ is said to be tangent to $\psi_\Ga$ at $u\in \fa$ if $\psi\in D_\Ga$ and $\psi(u)=\psi_\Ga(u)$.
We denote by $\dg$ the set of all linear forms tangent to $\psi_\Ga$ at
$\cal L\cap \inte\fa^+$, i.e.,
\be\label{dg} D_\Ga^\star:=\{ \psi\in D_\Ga:  \psi (u)=\psi_\Ga(u)\text{ for some $u\in \cal L\cap
\inte \fa^+$}\}.\ee 

For $\Gamma<\SO^\circ(n,1)$ and $\delta$ its critical exponent, we have $\dg=\{\delta\}$ and $D_\G=\{s\ge \delta\}$. 

Extending the construction of Patterson \cite{Pa} and Sullivan \cite{Su1}, Quint \cite{Quint} showed the following:
\begin{thm}\label{Q1} For any $\psi\in \dg$,  there exists a $(\Gamma, \psi)$-conformal measure supported on $\La$.
\end{thm}

\subsection*{Anosov subgroups.} 
Let $\Sigma$ be a finitely generated group.
For $\sigma\in \Sigma$, let $|\sigma|$ denote the word length of $\sigma$ for some
fixed symmetric generating set of $\Sigma$.
\begin{Def} (\cite{GW}, \cite{KLP}, \cite{GG}, \cite{BPS}) \label{Ano}
A representation $\pi: \Sigma \to G$  is Anosov with respect to $P$
if there exist a constant $c>0$
such that for all $\sigma\in \Sigma$ and
$\alpha\in \Pi$,
\be\label{dis} \alpha(\mu(\pi(\sigma)))\ge c |\sigma| -c.\ee 
 \end{Def}
 A discrete subgroup $\Ga<G$ is called an Anosov subgroup (with respect to $P$) if $\Gamma$ can be realized as the image 
$\pi(\Sigma)$ of an Anosov representation $\pi:\Sigma\to G$.
 If $\Ga=\pi(\Sigma)$ is Anosov, then $\Sigma$ is a Gromov hyperbolic group
 (\cite{KLP}, \cite{BPS}). 
As mentioned in the introduction, Anosov subgroups of $G$ were first introduced by Labourie for surface groups \cite{La}, and then extended by Guichard and Wienhard \cite{GW} to general word hyperbolic groups. Several equivalent characterizations have been  established, one of which is the above definition (see \cite{GG}, \cite{KLP}). 
When $G$ has rank one, the class of Anosov subgroups coincides with that of convex cocompact subgroups, and when $G$ is a product of two rank one simple algebraic groups, any Anosov subgroup arises in a similar fashion to \eqref{as}. 
Examples of Anosov subgroups include Schottky groups (cf. \cite[Def. 7.1]{ELO}), as well as Hitchin subgroups defined as follows. Let $\iota_d$ denote the irreducible representation $\PSL_2(\br)\to \PSL_d( \br)$, which is unique up to conjugations. A Hitchin subgroup is the image of a representation $\pi: \Sigma \to \PSL_d(\br)$ of a uniform lattice $\Sigma<\PSL_2(\br)$, which belongs to the same connected component as $\iota_d|\Sigma$ in the character variety $\text{Hom}(\Sigma, \PSL_d(\br))/\sim$ where the equivalence is given by conjugations. 

\medskip 
One of the important features of an Anosov subgroup is the following:
\begin{thm} \cite{PS}\label{PS} For any Anosov subgroup $\Ga<G$,
we have $$ \L\subset \inte \fa^+\cup\{0\}.$$
\end{thm}

\subsection*{Tempered representations.}
By definition, a
unitary representation of $G$ is a Hilbert space $\cal H_\pi$ equipped with a strongly continuous
homomorphism $\pi$ from $G$ to the group of unitary operators on $\cal H_\pi$.
 Given two unitary representations
$\pi$ and $\sigma$ of $G$, $\pi$ is said to be
weakly contained in  $\sigma$ if any diagonal
matrix coefficients of $\pi$ can be approximated, uniformly on compact sets, by convex
combinations of diagonal matrix coefficients of $\sigma$. 
We use the notation $\pi\propto\sigma$ for the weak containment.

The Harish-Chandra function $\Xi_G:G \rightarrow (0, \infty)$ is a bi-$K$-invariant function defined
via the formula
$$\Xi_G(g)=\int_K e^{-\rho(H(gk))} dk \quad\text{for all $g\in G$}$$  
where $dk$ denotes the probability Haar measure on $K$.
The following estimate is well-known, cf.\ e.g.\ \cite{Knapp1}: for any $\e>0$,
there exist $C, C_\e>0$ such that for any $g\in G$,
\be\label{hc} C e^{-\rho(\mu(g))} \le \Xi_G(g)\le C_\e e^{-(1-\e)\rho ( \mu(g))}.\ee 

\begin{Def} \label{te} A unitary representation $(\pi, \cal H_\pi)$ of $G$ is called {\it tempered}
if for any $K$-finite unit vectors $v, w\in \cal H_{\pi}$ and any $g\in G$,
$$|\langle \pi(g)v, w\rangle |\le (\op{dim}\langle Kv\rangle
\op{dim}\langle Kw\rangle)^{1/2} \Xi_G(g),$$
where $\langle Kv\rangle$ denotes the linear subspace of $\cal H_\pi$ spanned by $Kv$.
\end{Def}

\begin{prop} \cite{CHH}\label{chh}  The following are equivalent for a unitary representation $(\pi, \cal H_\pi)$ of $G$:
\begin{enumerate}
    \item $\pi$ is tempered;
    \item $\pi\propto L^2(G)$;
\item for any vectors $v, w\in \cal H_\pi$, the matrix coefficient $g\mapsto \langle \pi(g)v,w\rangle$ lies in $L^{2+\epsilon}(G)$ for any $\e>0$;
 \item  for any $\e>0$, $\pi$ is strongly $L^{2+\e}$, i.e., there exists a dense subset of $\cal H_\pi$
whose matrix coefficients all belong to $L^{2+\epsilon}(G)$.
\end{enumerate}
\end{prop}

In the whole paper, 
the notation $ f(v)\asymp g(v)$ means that the ratio $f(v)/g(v)$ is bounded uniformly between two positive constants, and $f\ll g$ means that $|f|\leq c |g|$ for some $c>0$.

\section{Positive joint eigenfunctions and conformal densities}\label{sje}
Let $G$ be a connected semisimple real algebraic group and
$\G<G$ be a Zariski dense discrete subgroup. The main goal of this section is to obtain Proposition \ref{d2}, which explains the relationship between positive
joint eigenfunctions on $\Gamma\ba X$ and $\Gamma$-conformal measures on the Furstenberg boundary of $G$. 

\subsection*{Joint eigenfunctions on $X$} Let $\cal D=\cal D(X)$ denote the ring of all $G$-invariant differential operators on $X$. We call a real valued function on $X$ a {\it joint eigenfunction} if it is an eigenfunction for all operators in $\cal D$. For each joint eigenfunction $f$, there exists an associated character $\chi_f:\cal D\to \br$ such that 
$$Df= \chi_f(D)f$$ for all elements $D\in\cal D$. The ring $\cal D$ is generated by $\rank(G)$ elements, and the set of all characters of $\cal D$ is in bijection with the space $\fa^*=\text{Hom}_\br(\fa, \br)$ modulo the action of the Weyl group, as we now explain.
Denote by $Z(\fg_{\CC})$ the center of the universal enveloping algebra $\Ug$ of $\fg_{\CC}$.  Recall the well-known fact that the joint eigenfunctions on $X$ can be identified with the right $K$-invariant real-valued $\Zg$-eigenfunctions  on $G$ (cf.\, \cite{Hel}).

 Letting $T$ be a maximal torus in $M$ with Lie algebra ${\mathfrak{t}}$, set $\fh=(\fa\oplus\mathfrak{t})$.
Then $\fh_{\CC}:=(\fa\oplus\mathfrak{t})_{\CC}$ is a Cartan subalgebra of $\fg_{\CC}$. We let $$\iota:\cal{Z}(\fg_{\CC})\rightarrow \scrS^{W}(\fh_{\CC})$$ denote the Harish-Chandra isomorphism from $Z(\fg_{\CC})$ to the Weyl group-invariant elements of the symmetric algebra $\scrS(\fh_{\CC})$ of $\fh_{\CC}$ \cite[Theorem 8.18]{Knapp1}. 

For any $\psi\in\fa^*$, we can extend it to $\fh$ by letting $\psi(J)=0$ for all $J\in\fm$, and then to $\scrS(\fh_{\CC})$ polynomially. This lets us define a character $\chi_{\psi}$ on $\cal{Z}(\fg_{\CC})$ by
\begin{equation}\label{Z}
\chi_{\psi}(Z):=\psi\big(\iota(Z)\big)
\end{equation}  
for all $Z\in \cal{Z}(\fg_{\CC})$. 
Conversely, if $f$ is a right $K$-invariant $\cal{Z}(\fg_{\CC})$-eigenfunction, then, since $\mathfrak{t}$ acts trivially on $f$, 
the associated character $\chi_f$ must arise as $ \psi\circ \iota$ for some $\psi\in \fa^*$.

\begin{ex}  \begin{itemize}
    \item 
Consider the hyperbolic space $\bH^{n}=\{(x_1,\cdots, x_{n-1}, y)\in \br^{n}:y>0\}$ with the metric $\frac{\sqrt{\sum_{i=1}^{n-1} dx_i^2+dy^2}}y$. The Laplacian 
$\Delta$ on $\bH^n$ is given as $\Delta=-y^2(\sum_{i=1}^{n-1}\frac{\partial^2}{\partial x_i^2} +\frac{\partial^2}{\partial y^2}  )+(n-2)y\frac{\partial}{\partial y}$
and the ring of $\op{SO}^\circ(n,1)$-invariant differential operators is generated by $\Delta$, i.e.,
a polynomial in $\Delta$. If $\psi\in \fa^*$ is given by $\psi(v)=\delta v$ for some $\delta\in \br$ under the isomorphism $\fa=\br$, then  $\chi_{\psi}(-\Delta)= \delta (n-1-\delta)$.

\item Let $G=\op{SO}^\circ(n_1,1)\times \op{SO}^\circ(n_2,1)$ and $X$ be 
the Riemannian product $\bH^{n_1}\times \bH^{n_2}$ for $n_1, n_2\ge 2$. Then
 $\cal D(X)$ is generated by the hyperbolic Laplacians $\Delta_1, \Delta_2$  on each factor $\bH^{n_1}$ and $\bH^{n_2}$. If we identify $\fa$ with $\br^2$ and
if a linear form $\psi\in \fa^*$ is given by $\psi (v)=\langle v, (\delta_1, \delta_2)\rangle$
 for some  vector $(\delta_1, \delta_2)\in \br^2$, 
then $\chi_{\psi}(-\Delta_i)= \delta_i (n_i-1-\delta_i)$ for $i=1,2$.
\end{itemize}
\end{ex}

\subsection*{Joint eigenfunctions on $\Gamma\ba X$}
We now consider joint eigenfunctions on $\Ga\ba X$  or, equivalently, $\Gamma$-invariant joint eigenfunctions on $X$.

\begin{Def} Let $\psi\in \fa^*$. Associated to a $(\Ga, \psi)$-conformal density $\nu=\{\nu_x:x\in X\}$ on $\cal F$, we define the following function $E_\nu$ on $G$: for $g\in G$,\be\label{ef} E_{\nu}(g):=|\nu_{g(o)}|=\int_{\scrF} e^{-\psi\big(H(g^{-1}k)\big)}\,d\nu_o([k]).\ee Since $|\nu_{\ga(x)}|=|\nu_x|$ for all $\ga \in \Ga$ and $x\in X$,
 the left $\Ga$-invariance and right $K$-invariance
 of $E_\nu$ are clear. Hence we may consider $E_\nu$
 as a $K$-invariant function on $\Gamma\ba G$, or, equivalently, as a
 function on $\Gamma\ba X$.
\end{Def}

\begin{prop}\label{conf}
For each $(\Gamma,\psi)$-conformal  density $\nu$ on $\cal F$, $E_\nu$ is
a positive joint eigenfunction on $\Ga\ba X$ with character $\chi_{\psi-\rho}$.  Conversely, any  positive  joint eigenfunction on $\Ga\ba X$ arises in this way for some $\psi\ge \rho$
and a $(\Ga, \psi)$-conformal density $\nu$ with $(\psi, \nu)$ uniquely determined.
\end{prop}

In order to prove this proposition, we consider
the following right $K$-invariant
function on $G$ for each $\psi\in \fa^*$ and $h\in G$:
\be\label{psih} \varphi_{\psi,h}(g)=e^{-\psi\big(H(g^{-1}h)\big)}\ee 
so that
$$E_\nu(g)=\int_{\scrF} \varphi_{\psi,k}(g) \,d\nu_o([k]).$$

We may also consider $\varphi_{\psi, h}$ as a function on $X$. Hence the first part of Proposition \ref{conf} is a consequence of the following:
\begin{lem}\label{Zeval}(\cite[Propositions 8.22 and 9.9]{Knapp1})
For any $\psi\in\fa^*$ and $h\in G$, the function $\varphi_{\psi,h}$
is a joint eigenfunction on $X$ with character $\chi_{\psi-\rho}$.
\end{lem}

\begin{proof}
While we refer to \cite{Knapp1} for the full proof, we outline some of the key points below, as we will use some part of this proof later. Since the elements of $\Zg$ commute with translation, we simply need to prove that $$[Z\varphi_{\psi,e}](e)=\chi_{\psi-\rho}(Z)\varphi_{\psi,e}(e)\quad\text{for any  $Z\in\Zg$};$$ the same identity will then hold for the function $g\mapsto \varphi_{\psi,e}(h^{-1}g)$, and thus also for $\varphi_{\psi,h}$ for any $h\in G$. Following \cite[Chapter VII]{Knapp1}, we define the (non-unitary) principal series representation $U^{\psi}$: for all $g\in G$ and $f\in C(K)$
$$ [U^{\psi}(g)f](k):=e^{-\psi\big(H(g^{-1}k)\big)}f\big(\kappa(g^{-1}k)\big)$$
for all $k\in K$. This extends to a representation $dU^{\psi}$ of $\Ug$ on the right $M$-invariant functions in $C^{\infty}(K)$ by way of the formula
$$ [dU^{\psi}(X)f](k)=\left. \frac{d}{dt}\right|_{t=0}[U^{\psi}\big(\exp(tX)\big)f](k) \quad\text{for any $X\in\fg$. }$$ Observe that $[Z\varphi_{\psi,e}](e)=[dU^{\psi}(Z)1](e)$, so in order to prove the proposition, it suffices to show that $dU^{\psi}(Z)=\chi_{\psi-\rho}(Z)$ for all $Z\in\Zg$.

The next key observation is that
$$ Z(\fg_{\CC})\subset \scrU(\fh_{\CC})\oplus \fn\,\scrU(g_{\CC}).$$
We thus write
$$ Z= Y+\sum_i X_i U_i,$$
where $Y\in \scrU(\fh_{\CC})$, $X_i\in \fn$, and $U_i\in\scrU(\fg_{\CC})$. Note that in this decomposition, $Y$ is uniquely defined. Now, for arbitrary $X\in\fn$ and $f$,
\begin{align*}
[dU^{\psi}(X)f](e)&= \left. \frac{d}{dt}\right|_{t=0} [U^{\psi}\big(\exp(tX)\big)f](e)=\left. \frac{d}{dt}\right|_{t=0} [U^{\psi}\big(\exp(tX)\big)f](e)
\\ =&\left. \frac{d}{dt}\right|_{t=0}e^{-\psi\big(H(\exp(-tX))\big)}f\big(\kappa(\exp(-tX))\big)=\left. \frac{d}{dt}\right|_{t=0}f(e)=0,
\end{align*}
so applying this to the $X_i$ and functions $dU^{\psi}(U_i)f$ gives
$$[dU^{\psi}(X_iU_i)f](e)=[dU^{\psi}(X_i)\big(dU^{\psi}(U_i)f\big)](e)=0,$$
hence $[dU^{\psi}(Z)f](e)=[dU^{\psi}(Y)f](e)$.
For $L\in \fm$, we have $f(\exp(-L))=f(e)$, so $[dU^{\psi}(J)f](e)=0$ for all $J\in\mathfrak{t}$. Thus, it is only the $\fa$ component of $Y$ that contributes to $[dU^{\psi}(Y)f](e)$. Finally, note that for $X\in\fa$, we have
\begin{align*} [dU^{\psi}(X)f](e)&=\left. \frac{d}{dt}\right|_{t=0}e^{-\psi\big(H(\exp(-tX))\big)}f\big(\kappa(\exp(-tX))\big)\\&=\left. \frac{d}{dt}\right|_{t=0}e^{t\,\psi(X)}f(e)=\psi(X)f(e).\end{align*}
Since the Harish-Chandra isomorphism consists of projection onto $\Uh$ and then composition with the ``$\delta$-shift" $H\mapsto H+\delta(H)1=H+\rho(H)1$,
where $\delta\in\mathfrak{h_{\mathbb{C}}^*}$ is the half-sum of the positive roots for $\fg_\CC$, this shows that $dU^{\psi}(Z)=\chi_{\psi-\rho}(Z)$.
\end{proof}

Letting $h=kan\in KAN$, we see that for any $g\in G$,
$$ \varphi_{\psi,h}(g)=e^{-\psi\big(H(g^{-1}h)\big)}=e^{-\psi\big(H(g^{-1}kan)\big)}=e^{-\psi\big(H(g^{-1}k)\big)}\cdot  e^{-\psi\big(\log(a)\big)},$$
i.e., the function $\varphi_{\psi,h}$ is a scalar multiple of $\varphi_{\psi,\kappa(h)}$. In fact, the functions $\varphi_{\psi,k}$, $k\in K$ form a complete set of minimal positive joint eigenfunctions with character $\chi_{\psi-\rho}$ with $\psi\geq \rho$, in the sense that if $f$ is a positive joint eigenfunction on $X$ with character $\chi_{\psi-\rho}$ such that $f\leq \varphi_{\psi,k}$ for some $k\in K$, then 
$$f= c\cdot \varphi_{\psi,k} $$ for some $c>0$ (cf.\ \cite{FU, Ka}, see also \cite[Theorem 1]{La}).

As a consequence, we have the following (cf. \cite[Theorem 3]{La}):

\begin{thm}\label{intrep} For any positive joint eigenfunction
$f$ on $X$, there exist $\psi\in \fa^*$ with $\psi\geq \rho$ and a Borel measure $\nu_o$ on $\cal F=K/M$ such that for all $g\in G$,
$$ f(g)=\int_{\scrF} \varphi_{\psi,k}(g)\,d\nu_o([k]). $$
Moreover, the pair $(\psi,\nu_o)$ is uniquely determined by $f$.
\end{thm}

\noindent{\bf Proof of the second part of Proposition \ref{conf}:} 
Let $f$ be a $\Ga$-invariant joint eigenfunction on $X$.
By Theorem \ref{intrep}, there exist unique $\psi\in \fa^*$ and
a Borel measure $\nu_o$ on $\F$ so that for all $g\in G$,
$$ f(g)=\int_{\scrF} \varphi_{\psi,k}(g)\,d\nu_o([k]).$$
Since $f$ is $\Gamma$-invariant, for any $\gamma\in\Gamma$,
\begin{align*}
f(g)&=f(\gamma g)=\int_{\scrF} \varphi_{\psi,k}(\gamma g)\,d\nu_o([k]) \\&=\int_{\scrF} \varphi_{\psi,\kappa(\gamma^{-1}k)}( g)\,e^{-\psi\big(H(\gamma^{-1}k)\big)}\,d\nu_o([k]) \\&=\int_{\scrF} \varphi_{\psi,\widetilde{k}}( g)\,e^{\psi\big(H(\gamma\widetilde{k})\big)}\,d\nu_o(\gamma\cdot[\widetilde{k}]).
\end{align*}
By the uniqueness of $\nu_o$ in the integral representation of $f$, 
$$d\nu_o([k])=e^{\psi\big(H(\gamma k)\big)}\,d\nu_o(\gamma\cdot[k]).$$
Hence $\nu=\{\nu_x\}$ is a $(\Gamma,\psi)$-conformal density on $\cal F$,  finishing the proof.

\medskip

We denote by $\psi_\Ga:\fa \to \br\cup\{-\infty\}$ the growth indicator function of $\Ga$ as defined in \eqref{grow3}.
\begin{thm} \label{d1} \cite[Theorem 8.1]{Quint}. Let $\Ga<G$ be Zariski dense.
If there exists a $(\Gamma,\psi)$-conformal measure on $\F$ for some $\psi\in \fa^*$, then
$$\psi\geq \psi_{\Gamma}.$$
\end{thm}

Therefore Proposition \ref{conf} and Theorem \ref{d1} yield the following:
\begin{prop} \label{d2} Let $\Gamma<G$ be a Zariski dense discrete subgroup.
If $\nu$ is a $(\Ga,\psi)$-conformal density for some $\psi\in \fa^*$, then $E_{\nu}$ is a positive joint
eigenfunction on $\Gamma\ba X$ with character $\chi_{\psi-\rho}$. Conversely, any positive joint eigenfunction  on $\Ga\ba X$ is of the form $E_{\nu}$ for some $(\Gamma,\psi)$-conformal density $\nu$  with $\psi\ge \max (\rho, \psi_\Ga)$, where $(\psi, \nu)$ is uniquely determined. \end{prop}

\section{Eigenvalues of positive eigenfunctions}
Let $\G$ be a torsion-free discrete subgroup of a connected semisimple real algebraic group $G$.
Let $\Delta$ denote the Laplace-Beltrami operator on $X$ or on $\Gamma\ba X$.
Since $\Delta$ is an elliptic differential operator, an eigenfunction is always smooth. We call a smooth function 
$\lambda$-harmonic if $$-\Delta f=\lambda f.$$ 
Let $\scrC\in\Zg$ denote the Casimir operator on $C^\infty(G)$ (or on $C^\infty(\Gamma\ba G)$)
whose restriction to
$K$-invariant functions coincides with $\Delta$.  Then $K$-invariant
$\cal C$-eigenfunctions on $\Ga\ba G$ correspond to Laplace eigenfunctions on $\Ga\ba X$.
In particular, a joint eigenfunction
of $\Gamma\ba X$ is a Laplace eigenfunction.

 Define the real number $\lambda_0=\lambda_0(\Gamma\ba X)\in [0,\infty)$ as follows: \be\label{ll} \lambda_0:= \inf\left\lbrace \frac{\int_{\Gamma\ba X}\|\text{grad} \, f \|^2\,d\vol}{\int_{\Gamma\ba X}|f|^2\,d\vol}\,:\,f\in C^\infty_c(\Gamma\ba X),\; f\neq 0 \right\rbrace .\ee

\subsection*{Positive Laplace eigenfunctions}
 \begin{thm}\cite[Theorem 2.1, 2.2]{Su} \label{no}
Suppose that $\Ga\ba X$ is not compact.
\begin{enumerate}
    \item 
For any $\lambda \le \lambda_0$, there exists a
positive $\lambda$-harmonic function on $\Gamma\ba X$;
\item For any $\lambda>\lambda_0$,
there is no positive $\lambda$-harmonic function on $\Gamma \ba X$.
\end{enumerate}

\end{thm}

We identity $\fa^*$ with $\fa$ via the inner product on $\fa$ induced by the Killing form on $\fg$. This endows an inner product on $\fa^*$. More precisely,
for each $\psi\in \fa^*$, there exist a unique $v_\psi\in \fa$ such that
$\psi =\la v_\psi, \cdot \ra$. Then $\la \psi_1, \psi_2\ra =\la v_{\psi_1}, v_{\psi_2}\ra$.
Equivalently, fixing an orthonormal basis $\lbrace H_i\rbrace$ of $\fa $, we have $\la \psi_1, \psi_2\ra =\sum_i \psi_1(H_i)\psi_2(H_i)$.

For $\psi\in \fa^*$, we set
\be\label{lp} \lambda_\psi:=\big(\|\rho\|^2-\|\psi-\rho\|^2\big) .\ee 

\begin{prop}\label{LaplaceEV} 
\begin{enumerate}
    \item  A positive joint eigenfunction on $X$ with character $\chi_{\psi-\rho}$, $\psi\in \fa^*$,  is $\lambda_\psi$-harmonic.
\item A positive Laplace eigenfunction on $X$ is
$\lambda_\psi$-harmonic for some $\psi\in \fa^*$ with $\psi\ge \rho$.

\end{enumerate}
\end{prop}
\begin{proof} 
Let $\psi\in \fa^*$. Recall the functions $\varphi_{\psi, h}$ in \eqref{psih}.
By Theorem \ref{intrep}, (1) follows if we show that for any $h\in G$,
\be\label{ccc} -\cal C \varphi_{\psi,h}=\lambda_\psi \varphi_{\psi,h}.\ee 
 Let $\{H_i\}$ be an orthonormal basis of $\fa$. To each $\alpha\in \Sigma$, let $H_\alpha\in \fa$ be the unique vector such that
$\alpha(x)=B(x, H_\alpha)=\langle x, H_\alpha\rangle$ for all $x\in \fa$, and choose a unit root vector $E_\alpha\in \mathfrak n$ so that $[x, E_\alpha]=\alpha(x) E_\alpha$ for all $x\in \fa$. 
We may write
$$ \scrC=\sum_i H_i^2+\sum_{\alpha\in \Sigma^+}\big(E_{\alpha}E_{-\alpha}+E_{-\alpha}E_{\alpha}\big)+J,$$
where $J\in \mathcal{U}(\fm_{\CC})$ (cf.\ \cite[Proposition 5.28]{Knapp2}). Now using $E_{-\alpha}E_{\alpha}=E_{\alpha}E_{-\alpha}-H_{\alpha}$ gives
$$ \scrC=\sum_i H_i^2-\sum_{\alpha\in\Sigma^+}H_{\alpha}+\sum_{\alpha\in \Sigma^+}2E_{\alpha}E_{-\alpha}+J.$$

 As in the proof of Lemma \ref{Zeval}, $[J\varphi_{\psi,h}](e)=0$, and $[E_{\alpha}E_{-\alpha}\varphi_{\psi,h}](e)=0$. Applying $-\scrC$ to $\varphi_{\psi,h}$ gives
\begin{align*} -\scrC\varphi_{\psi,h}&=-\left( \sum_i \psi(H_i)^2-\sum_{\alpha\in\Sigma^+}\psi(H_{\alpha})\right)\varphi_{\psi,h}\\&
=-\left(\|\psi\|^2-2\langle\rho,\psi\rangle \right)\varphi_{\psi,h}
\\&=\left( \|\rho\|^2-\|\psi-\rho\|^2\right)\varphi_{\psi,h},
\end{align*}
proving \eqref{ccc}.
Let $f$ be a positive $\lambda$-harmonic function on $X$, which we consider as a $K$-invariant function on $G$.
By \cite[Theorem 2]{La}, $f$ is of the form: for any $g\in G$,
$$f(g)=\int_{K/M\times \{\psi\ge \rho: \lambda_\psi=\lambda\} } \varphi_{\psi, k}(g)\; d\mu([k],\psi) $$
for some Borel measure $\mu$ on $K/M\times \{\psi\ge \rho: \lambda_\psi=\lambda\}$. By \eqref{ccc}, this implies (2). 
\end{proof}

\begin{cor} \label{bou} For any Zariski dense discrete subgroup $\Ga<G$,
 $$\sup\{\lambda_\psi : \psi\in D_\Ga^\star  \}\le \lambda_0 .$$
\end{cor}

\begin{proof} 
If $\Ga$ is cocompact in $G$, then $\psi_\Ga=2\rho$ and hence $\dg=\{2\rho\}$. Since $\lambda_0=0=\lambda_{2\rho}$ in this case, the claim follows.
 In general,
it follows from Theorem \ref{Q1} and Proposition \ref{d2} that for any $\psi\in \dg$, there exists a positive
joint eigenfunction on $\Ga\ba X$ with character $\chi_{\psi-\rho}$. Hence
the claim for the case when $\Ga$ is not cocompact in $G$ follows from Theorem \ref{no} and Proposition \ref{LaplaceEV}. 
\end{proof}

\section{Groups of the second kind and
positive joint eigenfunctions}
When $G$ has rank one (in which case the Furstenberg boundary is same as the geometric boundary of $X$),  a discrete subgroup $\G<G$ is said to be of the second kind if $\La\ne \F$. We extend this definition to higher rank groups as follows:

\begin{Def}\label{gp}
A discrete subgroup $\Gamma<G$ is of {\it the second kind} if there exists $\xi\in \F$ which is in general position with all points of $\La$, i.e., $(\xi, \La)\subset \F^{(2)}$.\end{Def}

Theorem \ref{no} provides a positive $\lambda$-harmonic function for any $\lambda \le \lambda_0$, when $\Ga\ba X$ is not compact. The following theorem can be viewed as a higher rank strengthening of this result.

\begin{thm} \label{exi}  Let $\Ga<G$ be of the second kind with
$\L\subset \inte \fa^+ \cup\{0\}$.
For any $\psi\in D_\Ga$, there exists a
positive joint eigenfunction on $\Ga\ba X$ with character $\chi_{\psi-\rho}$.
\end{thm}

By Proposition \ref{d2}, we get the following immediate corollary:
\begin{cor}\label{coexi}  Let $\Ga<G$ be of the second kind with
$\L\subset \inte \fa^+ \cup\{0\}$.
Then for any $\psi\ge\max (\psi_\Ga, \rho)$,
there exists a $(\Ga,\psi)$-conformal density.
\end{cor}
\begin{Rmk} \label{gpr}   \begin{enumerate} 
   \item Let $\Gamma_0<G$ be an Anosov subgroup.
Then any Anosov subgroup $\Gamma<\Gamma_0$ with $\Lambda_{\Ga_0}\ne \Lambda_{\Ga}$
is of the second kind. To see this, choose
any $\xi \in \Lambda_{\Ga_0}-\Lambda_{\Ga}$, and note that
$(\Lambda_{\Ga}, \xi)\subset \cal F^{(2)}$, since
any two distinct points of $\La_{\Ga_0}$ are in general position by the Anosov assumption on $\Ga_0$.
\item 
If $\Lambda\subset g Nw_0P$ for some $g\in G$, then
$( \Lambda, g^+)\subset \F^{(2)}$. One can construct many Schottky groups with 
$\Lambda\subset Nw_0P$, which would then be of the second kind.
\item  Let $G=\prod_{i=1}^k G_i$ be
a product of simple algebraic groups $G_i$ of rank one. Then $\F=\prod_i \F_i$ where $\F_i= G_i/P_i$.
Let
$\pi_i:\F\to \F_i$ denote the canonical projection. 
Then any discrete subgroup $\G<G$ such that $\pi_i(\La)\ne \F_i$ for all $i$ is of the second kind.
To see this,  it suffices to note that $(\La,\xi)\subset \F^{(2)}$ for any $\xi=(\xi_i)_i\in \F$ with $\xi_i \in \F_i -\pi_i(\La)$.
\item The well-known properties of the limit set of a Hitchin subgroup of
$\op{PSL}_d(\br)$ imply that Hitchin groups are not of the second kind for any even $d\ge 4$ 
or $d=3$; we thank Canary and Labourie  for communicating this with us.

\end{enumerate}
\end{Rmk}

For $ q\in X$ and $r>0$, we set 
$$B(q,r)=\{x\in X: d(x, q)\leq r\}.$$
For $p=g(o)\in X$,
the shadow of the ball $B(q,r)$ viewed from $p$ is defined as
 $$O_r(p,q):=\{(gk)^+\in \cal F: k\in K,\;
  gk\inte A^+o\cap  B(q,r)\ne \emptyset\}.$$

Similarly, for $\xi\in \F$, the shadow of the ball $B(q,r)$ viewed from $\xi$ is
defined by
$$O_r(\xi,q):=\{h^+\in \cal F: h\in G\text{ satisfies } h^-=\xi,\, ho\in  B(q,r)  \}.$$
We will use the following shadow lemma  to prove Theorem \ref{exi}:
\begin{lem} \cite[Lemma 5.6, 5.7]{LO} \label{sh} 
\begin{enumerate}
    \item If a sequence $q_i\in X$ converges to $\eta\in \cal F$, then for any $q\in X$, $r>0$ and $\e>0$,
$$O_{r-\e} (q_i, q)\subset O_r(\eta, q)\subset O_{r+\e}(q_i, q)$$
for all sufficiently large $i$.
\item There exists $\kappa>0$ such that for any $g\in G$ and $r>0$,
$$\sup_{\xi \in O_r(g(o),o)}\|\beta_{\xi}(g(o),o)-\mu(g^{-1})\|\le \kappa r .$$ 
\end{enumerate}
\end{lem}

\begin{lem}\label{com}  If $\L\subset \inte\fa^+\cup\{0\}$, then 
the union $\Ga(o)\cup \La$ is compact in the topology given in Definition \ref{conver}.
\end{lem}
\begin{proof}
The hypothesis implies that any sequence $\ga_i\to \infty$ in $\Ga$
tends to $\infty$ regularly, and hence has a limit in $\F$. Moreover the limit belongs to $\La$ by its definition. 
\end{proof}

\begin{lem}\label{an1} Suppose that
$\L\subset \inte\fa^+\cup\{0\}$. If $\xi\in \F$ satisfies that $(\xi,\La)\subset \F^{(2)}$,
then
 there exists $R>0$ such that 
 $$\xi\in \bigcap_{\ga\in \Ga} O_R(\ga (o), o).$$
\end{lem}
\begin{proof} 
We first claim that $\xi\in \bigcap_{\eta\in \La} O_R(\eta, o)$ for some $R>0$.
Note that $\lim_{R\to \infty} O_R(\eta, o)=\{z\in \F: (z, \eta)\in \F^{(2)}\}$. Hence for each $\eta\in \La$, 
we have $${R_\eta}:=\inf\{R+1: \xi\in O_{R}(\eta, o)\}<\infty.$$
It suffices to show that $R:=\sup_{\eta\in \La }R_{\eta}<\infty .$
Suppose not; then $R_{\eta_i}\to \infty$ for some sequence $\lbrace\eta_i\rbrace\subset \La$. By passing to a subsequence if necessary, we may assume that the $\eta_i$ converge to some $\eta$. From this it follows that $O_{R_\eta+1}(\eta, o) \subset  O_{R_\eta+2}(\eta_i, o)$
for all sufficiently large $i$.
Therefore $R_{\eta_i} \le R_{\eta} +3$, yielding a contradiction.

We now claim that $\xi\in \bigcap_{\ga\in \Ga} O_{R'}(\ga o, o)$ for some $R'>0$.
Suppose not; then there exist  sequences $\ga_i\to\infty$ in $ \Ga$
and $R_i\to \infty$ such that $\xi\notin O_{R_i}(\ga_i o,o)$.
By Lemma \ref{com}, after passing to a subsequence, we may assume that $\ga_i(o)$ converges to some $\eta\in \La$. By the first claim,
we have $\xi\in O_R(\eta, o)$.
By Lemma \ref{sh}, we have $\xi\in O_R(\eta, o)\subset O_{R+1}(\ga_i(o), o)$ for all sufficiently large $i$. This is a contradiction, since for $i$ large enough so that $R_i>R+1$, we have $\xi\notin O_{R+1}(\ga_i(o),o)$.
This proves the claim.
\end{proof}

As an immediate corollary of Lemmas \ref{sh} and
\ref{an1}, we obtain:
\begin{cor} \label{an2} 
If $\L\subset \inte\fa^+\cup\{0\}$ and $\xi\in \F$ satisfies that $(\xi,\La)\subset \F^{(2)}$, then
$$\sup_{\ga \in \Ga}\|\beta_\xi (\ga^{-1} o, o) -\mu (\ga)\|<\infty.$$
\end{cor}

\noindent{\bf Proof of Theorem \ref{exi}:} If $\psi\in \dg$, this follows from Theorem \ref{Q1}. Hence we assume $\psi\in D_\Ga -\dg$;
this implies that
\be\label{conv} \sum_{\ga\in \Ga} e^{-\psi(\mu(\ga))}<\infty \ee by
\cite[Lem. III. 1.3]{Quint2}.
As $\Ga$ is of the second kind, there exists $\xi\in \F$
such that $(\xi, \eta)\in \F^{(2)}$ for all $\eta\in \La$.
By Corollary \ref{an2},  we have $\sup_{\ga\in \Ga} \| \beta_\xi(\ga^{-1}o, o) - \mu(\ga)\|<\infty $. Therefore \eqref{conv} implies that
\be\label{conv2} \sum_{\gamma\in \Ga} e^{-\psi\big(\beta_\xi(\ga^{-1}o, o)\big)}<\infty. \ee

 For any fixed $x\in X$, we have
 $\beta_\xi(\ga^{-1}x, o)=\beta_\xi(\ga^{-1}o, o) +\beta_{\ga\xi}(x, o)$ and $\| \beta_{\ga\xi}(x, o)\|\le d(x,o)$. Hence   
 $e^{-\psi\big(\beta_\xi(\ga^{-1}o, o)\big)} \asymp   e^{-\psi(\mu(\ga))}$ with implied constant uniform for all $\ga\in \Ga$.

Therefore, by \eqref{conv} the following function $F_{\psi, \xi}$ on $X$ is well-defined: 
\be\label{fx} F_{\psi,\xi}(x):= \sum_{\gamma\in \Ga} e^{-\psi(\beta_\xi(\ga^{-1} x, o))} \quad\text{for $x\in X$}. \ee

If we write $\xi=[k_0]\in K/M=\cal F$,
then for any $g\in G$, $$\beta_{\xi}(\gamma^{-1} go, o)=\beta_M(k_0^{-1}\ga^{-1} go, o)=H(g^{-1}\ga k_0)$$ and hence
$e^{-\psi(\beta_{\xi}(\gamma^{-1} go, o))}=\varphi_{\psi, \ga k_0}(g).$
Therefore $$F_{\psi,\xi} =\sum_{\ga\in \Ga} \varphi_{\psi, \ga k_0} .$$
It now follows from Lemma 2.2 that $F_{\psi,\xi}$ is a positive $\Gamma$-invariant joint eigenfunction on $X$ with eigenvalue $\chi_{\psi-\rho}$. This finishes the proof.

\begin{Rmk} In the above proof, for any $\psi\in D_\Ga-\dg$ and  any $\xi\in \F$ with $(\La, \xi)\subset F^{(2)}$,  we have constructed a positive joint eigenfunction $F_{\psi,\xi}$ on $\Ga\ba X$ of eigenvalue $\chi_{\psi-\rho}$.
\end{Rmk}

Hence we get the following strengthened version of Corollary \ref{bou}:
\begin{cor} \label{noo} 
If $\Ga<G$ is of the second kind with
$\L\subset \inte \fa^+ \cup\{0\}$, then  \be\label{inq} \sup\{\lambda_\psi: \psi\in D_\Ga  \}\le \lambda_0.\ee 
\end{cor}

If $\Ga<\op{SO}^\circ(n,1)$ is a discrete subgroup with $\La\ne \partial\bH^n$, we have equality in \eqref{inq}, as was proved by Sullivan  \cite[Theorem 2.17]{Su}.

\section{The $L^2$-spectrum and uniqueness}\label{sje2}
Let $\G$ be a torsion-free discrete subgroup of a connected semisimple real algebraic group $G$.  The space $L^2(\Ga\ba X)$
consists of square-integrable functions together with the inner product $\la f_1, f_2\ra=\int_{\Ga\ba X} f_1 \bar f_2 \,d\vol$.

 Let $W^1(\Ga \ba X)\subset L^2(\Ga \ba X)$ denote the closure of $C_c^{\infty}(\Ga \ba X)$ with respect to the norm $\| \cdot\|_{W^1}$ induced by the inner product
$$\langle f_1,f_2\rangle_{W^1}:= \int_{\Ga \ba X} f_1\bar f_2\,d\vol + \int_{\Ga \ba X} \langle \op{grad}f_1, \op{grad}f_2\rangle\,d\vol  $$
for any $f_1, f_2\in W^1(\Ga\ba X)$.

 As $\Gamma\ba X$ is complete, there exists a unique self-adjoint operator on the space $W^1(\Ga\ba X)$ extending the Laplacian $\Delta$ on $C^\infty_c(\Gamma\ba X)$,
 which we also denote by $\Delta$.  The $L^2$-spectrum of $-\Delta$, which we denote by 
 $$\sigma(\Ga\ba X),$$ is the set of all $\lambda\in \c$ such that $\Delta+\lambda $ does not have a bounded inverse
 $(\Delta+\lambda)^{-1}:L^2(\Ga\ba X)\to W^1(\Ga\ba X)$. 
  The self-adjointness of $\Delta$ and the fact that $\langle -\Delta f, f\rangle=\int_X \|\op{grad}f\|^2d\vol$ for all $f\in C^\infty_c(\Gamma\ba X)$ imply $\sigma(\Ga\ba X)\subset [0,\infty)$.
  
  \medskip 
 
We will be using Weyl's criterion to determine $\sigma(\Ga\ba X)$:
 \begin{thm}(cf.\ \cite[Lemma 2.17]{Tes})\label{weyl} For $\lambda\in \br$,
we have $\lambda\in \sigma (\Ga\ba X)$ if and only if 
 there exists a sequence of unit vectors $F_n\in W^1(\Ga\ba X)$ such that
 $$\lim_{n\to \infty} \|(\Delta+\lambda)F_n\| = 0.$$
 \end{thm}

The number $\lambda_0=\lambda_0(\Ga\ba X)$ defined  in \eqref{ll}
is the bottom of the $L^2$-spectrum $\sigma(\Ga\ba X)$:

\begin{thm}\cite[Theorem 2.1, 2.2]{Su} \label{no3}
We have 
$$\lambda_0\in \sigma(\Ga\ba X)\subset [\lambda_0, \infty).$$
\end{thm}

Using Harish-Chandra's Plancherel formula, we can identify $\lambda_0(X) $ and $\sigma(X)$ for the symmetric space $X=G/K$:
\begin{prop}\label{ss}
 We have $\lambda_0(X)=\|\rho\|^2$ and $\sigma(X)=[\|\rho\|^2,\infty).$
\end{prop}
\begin{proof}
It is shown in \cite{Ka} that there are no positive Laplace eigenfunctions on $X$ with eigenvalue strictly bigger than $\|\rho\|^2$; hence
the inequality $\lambda_0 (X)\le \|\rho\|^2$ follows from Theorem \ref{no} for $\Gamma=\{e\}$. On the other hand, as seen in the proof of (1), $\varphi_{\rho,h}$ is a positive $\|\rho\|^2$-harmonic function (for any $h\in G$), hence $\lambda_0(X)=\|\rho\|^2$ by Theorem \ref{no}.
We now deduce the second claim  $\sigma(X)=[ \|\rho\|^2,\infty)$ from Harish-Chandra's Plancherel theorem (cf.\ e.g.\ \cite{Rosenberg}). For $\psi\in \fa^*$, define $\Phi_{\psi}\in C^{\infty}(K\ba G/K)$ by
$$ \Phi_{\psi}(g)= \int_K \varphi_{\rho+i\psi,k}(g)\,dk.$$
where $\varphi_{\rho+i\psi,k}(g)=e^{-(\rho+i\psi)\big(H(g^{-1}k)\big)}$.

Then by the same computation as \eqref{ccc}, we have $$-\mathcal{C}\Phi_{\psi}=-\Delta\Phi_{\psi}=(\|\rho\|^2+\|\psi\|^2)\Phi_{\psi}.$$

Given any $f\in C_c^{\infty}(\fa^*),$ we can define a function $F\in L^2(X)$ by the formula
$$ F(g)=\int_{\fa^*} f(\psi)\Phi_{\psi}(g)\,\frac{d\psi}{|\mathbf{c}(\psi)|^2};$$
here $d\psi$ denotes the Lebesgue measure on $\fa^*$ and $\mathbf{c}(\psi)$ denotes the Harish-Chandra $\mathbf{c}$-function.  The Plancherel formula says
$$ \|F\|_{L^2(X)}^2=\int_{\fa^*} |f(\psi)|^2\,\frac{d\psi}{|\mathbf{c}(\psi)|^2} $$
(see \cite{Rosenberg}). Let $\lambda \in [ \|\rho\|^2,\infty)$ be any number.
Choose $\psi_0\in\fa^*$ so that $\lambda=\|\rho\|^2+\|\psi_0\|^2$. We then choose a sequence of non-negative functions $\lbrace f_n\rbrace\subset C_c^{\infty}(\fa^*)$ with $\supp f_n\subset B_{1/n}(\psi_0)$ and $ \|F_n\|_{L^2(X)}=1$. 

Then
\begin{align*}(\Delta+\lambda)F_n&=\int_{\fa^*} f_n(\psi) (\Delta+\lambda) \Phi_{\psi}(g)\,\frac{d\psi}{|\mathbf{c}(\psi)|^2}\\
&= \int_{\fa^*} f_n(\psi) \big(\lambda - \|\rho\|^2-\|\psi\|^2\big) \Phi_{\psi}(g)\,\frac{d\psi}{|\mathbf{c}(\psi)|^2}.
\end{align*}

This gives
\begin{align*}\| (\Delta +\lambda)F_n\|_{L^2(X)}^2=&\int_{\fa^*} |\big(\lambda - \|\rho\|^2-\|\psi\|^2\big)f_n(\psi)|^2\,\frac{d\psi}{|\mathbf{c}(\psi)|^2}\\&\leq \max_{\psi\in B_{1/n}(\psi_0)} \left| \|\psi_0\|^2-\|\psi\|^2\right|^2 .
\end{align*}
Consequently, $$\lim_{n\rightarrow\infty}\| (\Delta +\lambda)F_n\|_{L^2(X)}=0.$$ By Weyl's criterion (Theorem \ref{weyl}), this implies that
$\lambda\in \sigma(X)$. This proves the claim.
\end{proof}

\begin{thm}\label{tlam} If $L^2(\Ga\ba G)$ is tempered, then $$\lambda_0(\Ga\ba X)=\|\rho\|^2.$$
\end{thm}
\begin{proof} 
Note that $\lambda_0=\lambda_0(\Ga\ba X)\le\lambda_0(X)= \|\rho\|^2$ by Proposition \ref{ss}.
Assume that $\lambda_0<\|\rho\|^2$. By Theorem \ref{weyl}, we can then find a $K$-invariant unit vector $f\in L^2(\Ga\ba G)_K$ such that
$$ \|(\Delta-\lambda_0) f\|<\frac{\|\rho\|^2-\lambda_0}{2}.$$
This gives
$$ \|\mathcal{C}f\|=\|\Delta f\| \leq \|(\Delta-\lambda_0) f\|+\lambda_0 <\frac{\|\rho\|^2+\lambda_0}{2}<\|\rho\|^2.$$
On the other hand, consider the direct integral representation of $L^2(\Ga\ba G)=\int_{\mathsf{Z}}^{\oplus} (\pi_{\zeta},\scrH_{\zeta})\,d\mu(\zeta) $ into irreducible unitary representations of $G$ which are tempered, by the hypothesis on the temperedness of $L^2(\Ga\ba G)$.
Hence
$$\|\scrC f\|^2 =\int_{\mathsf{Z}} \| d\pi_{\zeta}(\scrC)f_{\zeta}\|_{\zeta}^2\,d\mu(\zeta)\geq \left(\min_{\pi\;\mathrm{spherical\,tempered}} | d\pi(\scrC)|^2\right), $$
where $d\pi$ denotes the derived representation of $\mathcal{U}(\mathfrak{g}_{\CC})$ induced by $\pi$.
By Schur's lemma, there exists a character $\chi_{\pi}$ of $\Zg$ such that $d\pi(Z)=\chi_{\pi}(Z)$ for all $Z\in\Zg$. Moreover, for any spherical $\pi$, there exists $\psi_{\pi}\in\fa_{\CC}^*$ such that $\chi_{\pi}=\chi_{\psi_{\pi}}$ (cf. \eqref{Z}). Now, by Harish-Chandra's Plancherel formula (cf.\ e.g.\ \cite{Helg}), for any tempered spherical representation, we have
$$ \psi_{\pi}=\rho+i\op{Im}(\psi_{\pi}),$$
where $\op{Im}(\psi_{\pi})\in\fa^*$. As in the proof of Proposition \ref{ss}, we then obtain
$$ \chi_{\pi}(-\mathcal{C})=\|\rho\|^2+
\|\op{Im}(\psi_{\pi})\|^2.$$
 Thus for any spherical tempered representation $(\pi,\mathcal{H})$, we have $d\pi(\mathcal{C})\in\sigma(X)$ and hence, by Proposition \ref{ss},
$$\min_{\pi\;\mathrm{spherical\,tempered}} | d\pi(\scrC)|\geq\|\rho\|^2, $$
giving a contradiction.
\end{proof}

\begin{thm} \cite[Theorem 2.8 and Corollary 2.9]{Su} \label{atmost}
\begin{enumerate}
 \item Any positive Laplace eigenfunction in $L^2(\Ga\ba X)$  is $\lambda_0$-harmonic.
  \item If there exists a $\lambda_0$-harmonic function in $L^2(\Ga\ba X)$, then the space of $\lambda_0$-harmonic functions  in $\Ga\ba X$ is one-dimensional and generated by a positive function.
   \end{enumerate}
\end{thm}

\begin{proof}  Sullivan's proof in \cite{Su} uses the heat operator and superharmonic functions. We provide a more direct proof here.

Note that if $f\in L^2(\Ga \ba X)\cap C^{\infty}(\Ga \ba X)$ is a real-valued $\lambda$-harmonic function, then $f\in W^1(\Ga\ba X)$, since
$$ \int_{\Ga \ba X} \|\op{grad}f\|^2\,d\vol= -\int_{\Ga \ba X}  f\Delta f\,d\vol=\lambda \int_{\Ga \ba X} f^2\,d\vol. $$

The key fact for us is that $\lambda_0$ may also be expressed as an infimum over real-valued functions in $W^1(\Ga\ba X)$; for $f\neq 0$ in $W^1(\Ga\ba X)$, define $R(f)$ by 
$$ R(f)=  \frac{\|f\|^2_{W^1}}{\|f\|^2}-1 \geq 0 $$
where $\|\cdot\|$ denotes the $L^2(\Ga\ba X)$ norm. For any $f\neq 0\in W^1(\Ga\ba X)$, and all $\varphi$ with $\|f-\varphi\|_{W^1}$ small enough, we have
$$ \frac{\|\varphi\|_{W^1}-\|f-\varphi\|_{W^1}}{\|\varphi\|+\|f-\varphi\|_{W^1}}-1\leq R(f) \leq \frac{\|\varphi\|_{W^1}+\|f-\varphi\|_{W^1}}{\|\varphi\|-\|f-\varphi\|_{W^1}}-1, $$
i.e. $f\mapsto R(f)$ is continuous at each $f\neq0\in W^1(\Ga \ba X)$. The density of $C_c^{\infty}(\Ga \ba X)$ in $W^1(\Ga \ba X)$ then gives 
$$\lambda_0 = \inf_{\underset{f\neq 0}{f\in  C_c^{\infty}(\Ga\ba X)}} R(f)=\inf_{\underset{f\neq 0}{f\in  W^1(\Ga\ba X)}} R(f).$$

Now suppose that $\phi\in L^2(\Gamma\ba X)$ is a positive $\lambda$-harmonic function; so $\phi\in W^1(\Gamma\ba X)$.
We claim that $\lambda=\lambda_0$.  
By Green's identity, we have 
\begin{align*}\lambda_0 \leq R(\phi) = \frac{\int_{\Gamma\ba X}\|\text{grad} \, \phi\|^2\,d\vol}{\int_{\Gamma\ba X}|\phi|^2\,d\vol} =\frac{\int_{\Gamma\ba X}\phi(-\Delta \phi )\,d\vol}{\int_{\Gamma\ba X}|\phi|^2\,d\vol}= \lambda
\end{align*}
(cf.\ Proposition \ref{LaplaceEV}). On the other hand, since $\phi>0$, we have that for any $\varphi\in C_c^{\infty}(\Ga\ba X)$, 
\begin{align*}
\frac{\int_{\Gamma\ba X}\|\text{grad} \, \varphi \|^2\,d\vol}{\int_{\Gamma\ba X}| \varphi |^2\,d\vol}=\frac{\int_{\Gamma\ba X}\|\text{grad} \,\big(\phi\, \cdot\frac{\varphi}{\phi}\big) \|^2\,d\vol}{\int_{\Gamma\ba X}| \varphi |^2\,d\vol}.
\end{align*}
By Barta's identity \cite{Bar},
$$\int_{\Gamma\ba X}\|\text{grad} \,\big(\phi\, \cdot\sfrac{\varphi}{\phi}\big) \|^2\,d\vol= \int_{\Gamma\ba X} \phi^2\|\text{grad} \,\sfrac{\varphi}{\phi} \|^2\,d\vol -  \int_{\Gamma\ba X} \big(\sfrac{\varphi}{\phi}\big)^2\phi \Delta \phi\,d\vol  ,$$
so 
$$ \int_{\Gamma\ba X}\|\text{grad} \, \varphi \|^2\,d\vol \geq \int_{\Gamma\ba X} \big(\sfrac{\varphi}{\phi}\big)^2\phi (-\Delta \phi)\,d\vol=  \lambda \int \varphi^2\,d\vol,$$
i.e. 
$$\lambda  \leq   \frac{\int_{\Gamma\ba X}\|\text{grad} \, \varphi \|^2\,d\vol}{\int_{\Gamma\ba X}|\varphi|^2\,d\vol},$$
showing that $\lambda_0\geq\lambda  $. Hence $\lambda=\lambda_0$.

In order to prove (2), we first claim that $f\in W^1(\Ga\ba X)$ satisfies $-\Delta f= \lambda_0 f$ if and only if $R(f)=\lambda_0$. Suppose that $R(f)=\lambda_0$.  We will then show
that for any $\varphi\in C_c^\infty(\Ga\ba X)$, we have
\be\label {harmonic} \langle f,-\Delta\varphi\rangle{}=\lambda_0\langle f,\varphi\rangle{};\ee 
 this implies that $f$ is $\lambda_0$-harmonic.
Let $\varphi\in C_c^{\infty}(\Ga\ba X)$.
Since $R(f)=\lambda_0$, $f$ minimizes $R$. So for any $\varphi\in C_c^{\infty}(\Ga\ba X)$, the function $F:\RR\rightarrow\RR_{\geq 0}$ defined by $F(x)= R(f+x\varphi)$ has a local minimum at $x=0$, hence $F'(0)=0$. Now computing $F'(0)$ gives
$$ F'(0)=\frac{2\langle f,\varphi\rangle_{W^1}\|f\|^2-2\langle f,\varphi\rangle{}\|f\|_{W^1}^2}{\|f\|^4}=0.$$
From $R(f)=\lambda_0$, we obtain $\|f\|^2_{W^1}=(\lambda_0+1)\|f\|^2$, which, when entered into the identity above, gives
\begin{equation}\label{ipi}\langle f,\varphi\rangle_{W^1}=(\lambda_0+1)\langle f,\varphi\rangle{}.
\end{equation}
Letting $\lbrace f_i\rbrace_{i\in \NN} \subset C_c^{\infty}(\Ga\ba X)$ be a sequence converging to $f$ in $W^1(\Ga\ba X)$, Green's identity again gives
\begin{align}\label{Wlim}
  \notag\langle f,\varphi\rangle_{W^1}   =&\lim_{i\rightarrow\infty}  \langle f_i,\varphi\rangle_{W^1} =\lim_{i\rightarrow\infty} \int_{\Ga\ba X} f_i\varphi + \langle \op{grad}f_i,\op{grad}\varphi\rangle\,d\vol\\=&\lim_{i\rightarrow\infty} \int_{\Ga\ba X} f_i\varphi + f_i(-\Delta\varphi)\,d\vol=\langle f,\varphi\rangle{}+\langle f,-\Delta\varphi\rangle{}.
\end{align}
Combined with \eqref{ipi}, this gives $\langle f,-\Delta\varphi\rangle{}=\lambda_0\langle f,\varphi\rangle{}$  as in \eqref{harmonic}.

Conversely, if $f\in W^1(\Ga\ba X)$ satisfies $-\Delta f=\lambda_0 f$, then for any $\varphi\in C_c^{\infty}(\Ga\ba X)$, we have (as in \eqref{Wlim})
$$ \langle f,\varphi\rangle_{W^1}=\langle f,\varphi\rangle{}+\langle f,-\Delta\varphi\rangle{}=(\lambda_0+1)\langle f,\varphi\rangle{},$$
hence
$$ \|f\|^2_{W^1}=\sup_{\varphi\in C_c^{\infty}(\Ga \ba X)} \langle f,\varphi\rangle_{W^1}=\sup_{\varphi\in C_c^{\infty}(\Ga \ba X)} (\lambda_0+1)\langle f,\varphi\rangle{}=(\lambda_0+1)\|f\|^2,$$
giving $R(f)=\lambda_0$. This proves the claim.

Let $f\in W^1(\Ga\ba X)\cap C^{\infty}(\Ga\ba X)$ now be a real-valued $\lambda_0$-harmonic function. Then $|f|\in W^1(\Ga\ba X)$ and $R(|f|)=\lambda_0$. As shown above, $|f|$ is also a $\lambda_0$-harmonic function. 
Hence either $f$ is a constant multiple of $|f|$ or $f$ must change sign at some point $x_0$, hence $|f(x)|\geq |f(x_0)|=0$ for all $x\in \Ga \ba X$. However, since $\Delta |f|=-\lambda_0 |f| \leq 0$, the strong minimum principle (cf.\ e.g.\ \cite[Theorem 66, p. 280]{Petersen}) gives that if $|f|$ attains its infimum, then $|f|$ is in fact constant (in this case equal to zero). We therefore conclude that any $\lambda_0$-harmonic function in $L^2(\Ga\ba X)$ is a constant multiple of a positive function. This then implies that the space of $\lambda_0$-harmonic functions must be one-dimensional as two positive functions cannot be orthogonal to each other.
\end{proof}

The uniqueness in the above theorem has the following implications for joint eigenfunctions:
\begin{cor}\label{lzero}
\begin{enumerate}
    \item There exists at most one
positive joint eigenfunction in $L^2(\Ga\ba X)$ up to a constant multiple.
\item If there exists a positive joint eigenfunction in $L^2(\Ga\ba X)$ with character $\chi_{\psi-\rho}$, $\psi\in \fa^*$, then
$$\lambda_0 = \lambda_\psi .$$
\item  There exists a positive Laplace eigenfunction in $L^2(\Ga\ba X)$ if and only if
there exists a positive joint eigenfunction in $L^2(\Ga\ba X)$ of character $\chi_{\psi-\rho}$ with $\lambda_\psi=\lambda_0$.
\end{enumerate}
\end{cor}

\begin{proof} We only need to verify the third claim. Suppose that
$\phi\in L^2(\Gamma\ba X)$ is a postive Laplace eigenfunction.
Via the identification $L^2(\Gamma\ba X)=L^2(\Ga\ba G)_K$,
we may consider $\phi\in L^2(\Ga\ba G)_K$ as a positive
$\cal C$-eigenfunction for the Casimir operator $\cal C$. By Theorem \ref{atmost},
 $\cal C \phi =-\lambda_0\phi $. Let $D\in \cal Z(\mathfrak g_{\CC})$. Then $\cal C\circ D \phi= D \circ \cal C \phi =-\lambda_0 D\phi$. By the uniqueness in
Theorem \ref{atmost}, it follows that $D\phi$ is a constant multiple of $\phi$;
 and hence $\phi$ is an eigenfunction for $D$ as well. Therefore $\phi$ is a joint eigenfunction.\end{proof}

\noindent{\bf Spherical unitary representations contained in $L^2(\Ga\ba G)$.}
We let $C_c(G//K)$ denote the Hecke algebra of $G$, i.e.
$$C_c(G//K)=\lbrace f\in C_c(G)\,:\, f(k_1gk_2)=f(g)\qquad \text{for all
$g\in G,\;k_1,k_2\in K$}\rbrace.
$$

Each element of $C_c (G // K)$ acts on $C(G)$ via right convolution $*$.
\begin{lem} \label{hecke}
A positive $K$-invariant joint
eigenfunction on $G$ is an eigenfunction for the action of the Hecke algebra. More precisely, if 
\be\label{joint} \phi(g)=\int_{\scrF} \varphi_{\psi,k}(g)\,d\nu_o([k]), \quad\text{$g\in G$},\ee
for some $\psi\in \fa^*$ and a $(\Ga, \psi)$-conformal measure $\nu_o$ on $\F=K/M$, then for all $f\in C_c(G//K)$,
$$(\phi *f)(g)  =\left(\int_{G}f(h) e^{-\psi\big(H(h)\big)}\,dh\right)\phi(g).$$
\end{lem}
\begin{proof} 
Given $f\in C_c(G//K)$, we have
\begin{align*}
\big( \phi\ast f \big)(g)&=\int_G \phi(gh^{-1})f(h)\,dh=\int_{G}\int_{\scrF}\varphi_{\psi,k}(gh^{-1})f(h) \,d\nu_o([k])\,dh\\&=\int_{\scrF}\int_{G}f(h) e^{-\psi\big(H(hg^{-1}k)\big)}\,dh\,d\nu_o([k]).
\end{align*}
Now using $H(hg^{-1}k)=H(h\kappa(g^{-1}k))+H(g^{-1}k)$ and then the change of variables $h'=h\kappa(g^{-1}k)$ gives
\begin{align*}
\big( \phi\ast f \big)(g)&= \int_{\scrF}\left(\int_{G}f\big(h\kappa(g^{-1}k)^{-1}\big) e^{-\psi\big(H(h)\big)}\,dh\right)e^{-\psi\big(H(g^{-1}k)\big)}\,d\nu_o([k])
\\&=\int_{\scrF}\left(\int_{G}f(h) e^{-\psi\big(H(h)\big)}\,dh\right)e^{-\psi\big(H(g^{-1}k)\big)}\,d\nu_o([k])\\&=\left(\int_{G}f(h) e^{-\psi\big(H(h)\big)}\,dh\right)\phi(g),
\end{align*}
since $f\in C(G//K)$, and is thus right $K$-invariant. In total, we have shown that $\phi$ is an eigenfunction of the $f$-action, with eigenvalue $\int_{G}f(h) e^{-\psi\big(H(h)\big)}\,dh$.\\
\end{proof}

\begin{thm}\label{rep}
If $\phi\in L^2(\Ga\ba G)_K$ is a positive Laplace eigenfunction of norm one, there exists  a unique irreducible spherical unitary subrepresentation $(\pi, \scrH_{\phi})$ of $L^2(\GaG)$, and $\phi$ is the unique $K$-invariant unit vector in $\scrH_{\phi}$. 
\end{thm}
\begin{proof} 
By Corollary \ref{lzero}, $\phi$ is given by \eqref{joint}
for some $\psi\in \fa^*$.
 Define $\Phi:G\rightarrow\CC$ by 
$$ \Phi(g):=\langle g.\phi,\phi\rangle$$
for all $g\in G$ where the $g$ action on $L^2(\Gamma\ba G)$ is via the translation action of $G$ on $\Gamma\ba G$ from the right. Given $f\in C_c(G//K)$, we then have, using Lemma \ref{hecke},
\begin{align*}
\big( \Phi\ast f \big)(g)&=\int_G \Phi(gh^{-1})f(h)\,dh=\int_{G} \langle (gh^{-1}).\phi,\phi\rangle f(h)\,dh\\&=\int_{G} \langle f(h) h^{-1}.\phi,g^{-1}.\phi\rangle \,dh=\big\langle \phi\ast f, g^{-1}.\phi\big\rangle
\\&=\left( \small\int_{G}f(h) e^{-\psi\big(H(h)\big)}\,dh\right)\Phi(g),
\end{align*}
i.e. $\Phi$ is also a $C_c(G//K)$-eigenfunction. Also note that $\Phi(e)=1$, and since $\phi$ is right $K$-invariant, $\Phi$ is \emph{bi}-$K$-invariant. Moreover, being the matrix coefficient of a unitary representation, $\Phi$ is also positive definite, i.e., for any
$g_1, \cdots, g_n\in G$ and $z_1, \cdots, z_n\in \mathbb C$,
$$\sum_{1\le i,j\le n} z_i \bar{z_j}
\Phi(g_j^{-1} g_i)\ge 0.$$
We have thus shown that $\Phi$ is a positive definite \emph{spherical} function. Letting $\scrH_{\phi}$ denote the closure of $\mathrm{span}\lbrace g.\phi\,:\,g\in G\rbrace$ in $L^2(\GaG)$, by \cite[Chapter IV§5, Corollary of Theorem 9]{Lang}, $ \scrH_{\phi}$ is an irreducible (spherical) unitary subrepresentation of the quasi-regular representation $L^2(\GaG)$. The uniqueness follows from Corollary \ref{lzero}.
\end{proof}

We require the following lemma in the proof of Theorem \ref{nod}:
\begin{lem}\label{max}
Let $\psi \geq \rho$ and $\psi\not\in \RR\rho$. Denote by $\psi'$ be the element of the line $\RR \psi$ closest to $\rho$. Then $\psi'\not \geq \rho$.
\end{lem}
\begin{proof} Let $\phi:=\psi-\rho$. Note that $\phi \ge 0$ on $\fa$ by the hypothesis.
Then 
\begin{align*}
\psi'=\frac{\langle \psi,\rho\rangle}{\|\psi\|^2}\psi=\frac{\langle \rho+\phi,\rho\rangle}{\|\rho+\phi\|^2}\psi=\left( 1-\frac{\|\phi\|^2+\langle\rho,\phi\rangle}{\|\rho+\phi\|^2}\right)\psi,
\end{align*}
i.e. $\psi'=t \psi$ with $0<t<1$. Now, if $\psi'\geq \rho$, we could repeat the process with $\psi'$ in place of $\psi$ to find another, \emph{different}, closest vector in $\RR \psi$  to $\rho$, which is not possible.
\end{proof}

\begin{thm}\label{nod} Let $\Ga<G$ be of the second kind with
$\L\subset  \inte \fa^+\cup\{0\}$.
If there exists a $\lambda_0$-harmonic function in $L^2(\Ga\ba X)$, then $$\lambda_0= \lambda_\psi$$ for some
$\psi\in \dg\cup\{\rho\}$.
\end{thm}
\begin{proof}

Suppose that
$\psi\in D_{\Gamma}\setminus (\lbrace \rho\rbrace\cup D_{\Gamma}^\star)$ and that $\psi\ge \rho$. Assume that
there exists a positive joint eigenfunction $\phi\in L^2(\Ga\ba X)$ with character $\chi_{\psi-\rho}$. By Corollary \ref{lzero},
\be\label{no2} \lambda_0=\lambda_\psi=\|\rho\|^2-\|\psi-\rho\|^2.\ee 
Since $\psi_\Ga$ is  concave, there exists $0<c\le 1$ such that $c\psi (u)=\psi_\Ga(u)$
for some $u\in \L$. So $\psi_0:=c\psi\in D_{\Ga}^\star$.
Since $\psi\not\in \dg$, we have $0<c<1$. 
There exists a unique $s_0\in \br$ such that
\be\label{so2} \|s_0\psi_0-\rho \|=\min\{\|s\psi-\rho \|: s\in \br\},\ee  that is,
$s_0 \psi_0$ be the element on the line $\br \psi$ that is closest to $\rho$.

We claim that $s_0 c\le 1$; since $0<c<1$, this implies that $\max\{1, s_0\}<c^{-1}$.
If $\psi\in\RR\rho$, then $s_0\psi_0=\rho$. Since $\psi_0=c\psi$, we get $s_0c\psi=\rho$. By the hypothesis $\rho\le \psi$, $s_0c\le 1$.
Now suppose $\psi\notin \br \rho$. Assume that $s_0c>1$. Then $s_0\psi_0=
s_0c\psi> \psi$. Hence $s_0c\psi\in D_\Ga$. By Corollary \ref{noo}
and \eqref{no2}, we get 
$$\|s_0c\psi-\rho\| \ge \|\psi-\rho\|.$$ By the choice of $s_0$ in \eqref{so2}, it follows that $\|s_0c\psi-\rho\|=\|\psi-\rho\|$.
Since $s_0c\psi> \psi\ge \rho$, this yields a contradiction. Therefore the claim $s_0c\le 1$ follows.

We now choose $t$ so that $\max\{1, s_0\}<t<c^{-1}$.
Since $t>1$ and $\psi_0\in \dg$, $t\psi_0\in D_\Ga$.
Note also that $s\mapsto \lambda_{s\psi_0}$ is strictly decreasing on the interval $[s_0,\infty)$.  Since $s_0<t<c^{-1}$ and $c^{-1}\psi_0=\psi$, we get 
$$ \lambda_0=\lambda_{\psi}<\lambda_{t\psi_0}. $$
This contradicts Corollary \ref{noo}. This implies the claim by Corollary \ref{lzero}. \end{proof}

If we use the norm on $\frak{so}(n,1)$
which endows the constant curvature $-1$ metric on $\bH^n$, then for any non-elementary discrete subgroup
$\G<\op{SO}^\circ(n,1)$, $\dg=\{\delta\}$ and hence the above theorem says that if a $\lambda_0$-harmonic function belongs to $L^2(\Ga\ba \bH^n)$, then $\lambda_0$ must be given by either $\delta(n-1-\delta)$
or $\frac{1}{4}(n-1)^2$.

\section{Smearing argument in higher rank}\label{ssm}
Let $\G$ be a torsion-free discrete subgroup of a connected semisimple real algebraic group $G$. 
The goal of this section is to prove the following: 
\begin{thm}\label{tr} 
If $\L\ne\fa^+$, then no positive joint eigenfunction belongs to $L^2(\Gamma\ba X)$. 
\end{thm}

Combined with Corollary \ref{lzero}, we get the following corollary which implies Theorem \ref{m33}(4) in higher rank.

\begin{cor}\label{ttt} 
If $\L\ne\fa^+$, 
there exists no positive Laplace eigenfunction in $L^2(\Ga\ba X)$.
In particular,
if $\text{rank }G\ge 2$ and $\Ga<G$ is Anosov, 
no positive Laplace eigenfunction 
belongs to $L^2(\Gamma\ba X)$.\end{cor}
 The second part follows from the first by Theorem \ref{PS}.
Theorem \ref{tr} will be deduced from Theorem \ref{smear}, the proof of which is based
on the smearing argument of Thurston and Sullivan (see \cite{Su2}, \cite{CI} and also \cite{Su3} for historical remarks and the origin of the name ``smearing argument"). We also refer to \cite[Theorem 3.1]{RT}.

\begin{Def}[Hopf parameterization]\label{hopf} \rm The homeomorphism 
  $G/M\to \F^{(2)} \times \mathfrak a$
  given by $gM \mapsto (g^+, g^-, b=\beta_{g^-}(e,g)) $
  is called the Hopf parameterization of $G/M$.
\end{Def}

Fix   a pair of linear forms $\psi_1, \psi_2\in \mathfrak a^*$. For $x\in X$ and $(\xi, \eta)\in \F^{(2)}$,
 let 
 \be\label{phi} \phi_x (\xi, \eta)= e^{\psi_1\big( \beta_\xi(x, go)\big)+ \psi_2\big( \beta_{\eta}(x, go)\big)},\ee 
 where $g\in G$ is such that $g^+=\xi$ and $g^-=\eta$.
 
Let $\nu=\{\nu_x: x\in X\}$ and $\bar\nu=\{\bar \nu_x: x\in X\}$
 be respectively $(\Gamma, \psi_1)$ and $(\Gamma, \psi_2)$-conformal densities on $\cal F$.
Using the Hopf parametrization, we define the following locally finite Borel measure $\tilde m_{\nu, \bar \nu}$ on $G/M$: for $(\xi,\eta, v)\in \F^{(2)}\times \fa$, 
\begin{equation}\label{bms}
d\tilde m_{\nu, \bar\nu}(\xi, \eta, v)
={\phi_{ x} (\xi, \eta)} d\nu_{x} (\xi) d\bar \nu_{x}(\eta) dv
\end{equation}
  where $dv$ is the Lebesgue measure on $\mathfrak a$ and $x\in X$ is any element;
  it follows from the $\Ga$-conformality
  of $\{\nu_x\}$ and $\{\bar\nu_x\}$ that this definition is independent of $x\in X$. The measure $\tilde m_{\nu, \bar\nu}$ is left $\G$-invariant and right $A$-semi-invariant:
  for all $a\in A$,
\be\label{semi}
a_*\tilde m_{\nu, \bar{\nu}}=e^{(-\psi_1+\psi_2\circ \i)(\log a)}\,\tilde m_{\nu, \bar{\nu}} .\ee 
Note that $\psi_2=\psi_1\circ \i$ if and only if $\tilde m_{\nu, \bar \nu}$ is $A$-invariant.
We denote by $m_{\nu, \bar\nu}$
the $M$-invariant Borel measure on $\Ga\ba G$ induced by
$\tilde m_{\nu,\bar\nu}$; this measure is called the (generalized) Bowen-Margulis-Sullivan measure associated to the pair $(\nu, \bar\nu)$ \cite{ELO}.

\begin{thm}[Smearing theorem] 
\label{smear} For any pair $(\nu, \bar\nu)$ of $\Ga$-conformal densities on $\F$, there exists $c>0$ such that
$$m_{\nu,\bar{\nu}}(\Ga\ba G) \le 
c \int_{\text{$one$-neighborhood of $\op{supp}m_{\nu, \bar\nu}$}} E_{\nu}(x) E_{\bar{\nu}}(x) \, d\vol(x)  .$$
\end{thm}

\begin{proof} Let $Z=G/K\times \cal F^{(2)}$. For any $(\xi, \eta)\in \F^{(2)}$, we write $[\xi, \eta]=gAo\subset X$ for any $g\in G$ such that $g^+=\xi$ and $g^-=\eta$;  $[\xi, \eta]$ is a maximal flat in $X$ defined independently of the choice of $g\in G$. Let $\psi_1, \psi_2\in \fa^*$ be linear forms such that
$\nu$ and $\bar\nu$ are respectively $(\Ga, \psi_1)$ and $(\Ga, \psi_2)$-conformal densities. Let $\phi_x$ be defined as in \eqref{phi}
 for  all $x\in X$. We also denote by $W_{\xi, \eta}\subset X$ the one neighborhood of $[\xi, \eta]$.
 Consider the following locally finite Borel measure
 $\alpha$ on $Z$ defined as follows: for any $f\in C_c(Z)$, 
 $$\alpha (f)=\int_{(\xi, \eta)\in \F^{(2)}}\int_{z\in W_{\xi, \eta}} f(z, \xi, \eta) \, dz \, dm(\xi, \eta),$$
 where  $dz$ is the $G$-invariant measure on $X$, and
 $$dm (\xi, \eta)=\phi_x(\xi, \eta)
 d\nu_x(\xi) d\bar\nu_{ x}(\eta)$$ (observe that this definition is independent of $x$).

Consider the natural diagonal action of $\G$ on $Z$. Since
$dz$ and $dm$ are both left $\Ga$-invariant, $\alpha$ is also left $\Ga$-invariant and hence induces a measure on the quotient space $\Gamma\ba Z$, which we also denote by $\alpha$ by abuse of notation.

Define the projection $\pi':Z\to G/M$  as follows: for
 $(x, \xi, \eta)\in X \times \F^{(2)}$, choose $g\in G$ so that $g^+=\xi$ and $g^-=\eta$. Then there exists a unique element $a\in A$
 such that 
 $$d(x, gao)=d(x, gA o)=\inf_{b\in A} d(x, gbo);$$ this follows from \cite[Proposition 2.4]{BH} since $X$ is a $\rm{CAT} (0)$ space and $gA(o)$ is a convex complete subspace of $X$. In other words, the point
 $gao$ is the orthogonal projection of $x$ to the flat $[\xi, \eta]=gAo$.
 We then set $$\pi'(x, \xi, \eta)= gaM\in G/M;$$
 this is well-defined independent of the choice of $g\in G$. Noting that  $\pi'$ is $\Ga$-equivariant, we denote by
 $$\pi: \text{supp}(\alpha )\subset \Gamma\ba Z \to \text{supp }{ (m_{\nu,\bar\nu})}\subset \Gamma\ba G/M$$ the map induced by  $\pi'$. Fixing $[ga]\in \Ga\ba G/M$, the fiber $\pi^{-1}[ga]$
 is of the form $[( ga D_0, g^+, g^-)]$, where \begin{multline*} D_0=\{s\in X: d(s, o)\le1,\\
 \text{the geodesic connecting $s$ and $o$ is orthogonal to
 $Ao$ at $o$}\}.\end{multline*}

 Noting that each fiber $\pi^{-1} (v)$, $v\in\supp m_{\nu,\bar\nu}$, is isometric to $D_0$,
 we have for any Borel subset
 $S\subset \supp m_{\nu,\bar\nu}$,
\be\label{mu} \alpha (\pi^{-1} (S))=\op{Vol}(D_0)\cdot m_{\nu,\bar\nu}(S); \ee the volume of $D_0$ being computed with respect to the volume form induced by the $G$-invariant measure on $X$. Consider now the map $p: \text{supp} (\alpha) \to \Gamma\ba X$ defined by $p([(z,\xi, \eta)])=[z]$
 for any $(z,\xi,\eta)\in \text{supp} (\alpha)$.
 Let $F=\pi^{-1} (\supp m_{\nu,\bar\nu})\subset \text{supp}(\alpha)$. We write 
 $$\alpha (F)=\int_{\Gamma\ba X} \alpha_x(p^{-1}(x)\cap F)\, dx,$$ 
 where $\alpha_x$ is a conditional measure on the fiber $p^{-1}(x)$. We claim that there exists a constant $c>0$
 such that for any $x\in \Gamma\ba  X$,
 \be\label{al} \alpha_x(p^{-1}(x))\le c  \, E_\nu (x) \cdot  E_{\bar\nu} (x) .\ee
Since $p^{-1}(x)\cap F=\emptyset$ for $x$ outside
of the one neighborhood of $\op{supp}(m_{\nu, \bar\nu})$,
this together with \eqref{mu}, implies that
 $$\op{Vol}(D_0) \cdot |m_{\nu, \bar\nu}|= \alpha (F)\le c\cdot   \int_{\text{one neighborhood of $\op{supp}(m_{\nu, \bar\nu})$}} E_\nu(x) E_{\bar\nu}(x) \, dx $$
 finishing the proof. Note that for any $h\in G$,
$$V_{ho}:=\{( \xi, \eta)\in \F^{(2)}: [\xi, \eta]\cap B(ho,1)\ne \emptyset\}$$ is a compact subset
 of $\cal F^{(2)}$; if $\lbrace g_i\rbrace \subset G$ and $\lbrace a_i\rbrace\subset A$ are sequences such that
 $d(g_ia_i o, ho)\le 1$, then (by passing to a subsequence) we may assume that $g_ia_i$ converges to some $g_0\in G$. This implies $(g_i^+, g_i^-)\to (g_0^+, g_0^-)\in  \cal F^{(2)}$ as $i\to \infty$ and $d(g_0o,ho)\leq 1$, from which the compactness of $V_{ho}$ follows.
 It follows that
 $$c:=\sup \{ \phi_{o}(\xi,\eta) : (\xi, \eta)\in V_{o}\} <\infty .$$ 
 
 By the equivariance
 $\phi_{ho}(\xi, \eta)=\phi_o(h^{-1}\xi,h^{-1}\eta)$,
 we have for any $h\in G$,
 $$\sup \{ \phi_{ho}(\xi,\eta) : (\xi, \eta)\in V_{ho}\}=c .$$

 Note that if $x= [ho]\in \Gamma\ba X$ for $h\in G$, then
 $$p^{-1}(x)=\{[(ho, \xi, \eta)]\in \supp (\alpha): [\xi, \eta]\cap B(ho,1)\ne \emptyset\}\simeq V_{ho}.$$

Therefore for any $x =[ ho] \in \Gamma\ba X$,
 \begin{align*}
\alpha_x(p^{-1}(x)) &= \alpha_x(V_{ho})\\
&=\int_{(\xi, \eta)\in V_{ho}} \phi_{ho}(\xi, \eta)\, d\nu_{ho}(\xi) d\bar\nu_{ho}(\eta) \\
 &\le {c}   \int_{(\xi, \eta)\in V_{ho}} d\nu_{ho}(\xi) d\bar\nu_{ho}(\eta)\\ &
 \le  {c}\cdot |\nu_{ho}|\cdot |\bar\nu_{ ho}| = {c}\cdot E_\nu (x)\cdot E_{\bar\nu}(x).
 \end{align*}
 This proves \eqref{al}, and hence finishes the proof.
 \end{proof}

\noindent{\bf Proof of Theorem \ref{tr}.} Suppose
that $\phi\in L^2(\Ga\ba X)$ is a positive joint eigenfunction. 
By Proposition \ref{d2}, $\phi=E_\nu$
for some $(\Ga, \psi)$-conformal density $\nu$. We may form the $MA$-semi-invariant measure $m_{\nu, \nu}$, and apply Theorem \ref{smear}.
Since $E_\nu\in L^2(\Gamma\ba G)$, it follows that $m_{\nu, \nu}(\Ga\ba G)<\infty$.
The finiteness of $|m_{\nu, \nu}|$ implies that
 $m_{\nu, \nu}$ is indeed $MA$-invariant by \eqref{semi} and it is conservative for
any one-parameter subgroup of $A$. In particular, for any non-zero $v\in \fa^+$,
there exist $g\in G$, sequences $t_i\to +\infty$ and $\gamma_i\in \Ga$ such that the sequence
$\gamma_i g \exp (t_i v) $ is convergent. This implies that
$\sup_i \|t_i v -\mu(\ga_i^{-1})\|<\infty$ and hence $t_i^{-1}
\mu(\ga_i^{-1})$ converges to $v$, and hence $v\in \L$.
Therefore $\L=\fa^+$. 
This finishes the proof.

\begin{Rmk}
If $\G<G$ is Zariski dense and $\psi >\psi_\Ga$, then for any $(\Ga, \psi)$-conformal density $\nu$,
$E_\nu\notin L^2(\Gamma\ba X)$.
To see this, note that by  \cite[Lem. III. 1.3]{Quint2}, the condition $\psi>\psi_\Ga$ implies that $$\sum_{\ga\in \Ga}e^{-\psi (\mu(\ga))} <\infty.$$  On the other hand,
by Theorem 1.4 of \cite{BLLO}, the finiteness of
$m_{\nu, \nu}$ implies that
$\sum_{\ga\in \Ga}e^{-\psi (\mu(\ga))} =\infty$. Hence we must have $|m_{\nu,\nu}|=\infty$. Then the claim follows from
Theorem \ref{smear}.
\end{Rmk}

\section{Injectivity radius and  $L^2(G)\propto L^2(\Ga\ba G)$}\label{injs}
As before let $G$ be a connected semisimple real algebraic group. Recall
from Proposition \ref{ss} that $\sigma(X)=[\|\rho\|^2, \infty)$.
In this section, we prove the following: 
\begin{thm} \label{iinj} Let $\Ga< G$ be an Anosov subgroup. We suppose that $\Ga$ is not a cocompact lattice in a rank one group $G$.
Then
$$L^2(G)\propto L^2(\Ga\ba G)\quad\text{ and }\quad \sigma(X)=[\|\rho\|^2, \infty)\subset \sigma(\Ga\ba X).$$
\end{thm}

Note that if $\Ga<G$ is Anosov, $\Ga\ba G$ has infinite volume except when $\Ga$ is a cocompact lattice in a rank one group $G$. The latter case has to be ruled out from the above theorem since the conclusions are not true in that case; $L^2(\Ga\ba G)$ contains the constant function and $\sigma(\Ga\ba X)$ is countable.
When $G=\SO^\circ(n,1)$, an Anosov subgroup $\Ga<G$ is simply a convex cocompact subgroup, in which case this theorem is well-known due to the work of Lax and Phillips \cite{LP}.

We will need the following lemma: when $G$ is of rank one, we may write $A=\{a_t:t\in \br\}$ as a one-parameter subgroup, and a loxodromic element
$g\in G$ is of the form $g= h a_tm h^{-1}$ for some $t\ne 0$, $m\in M$ and $h\in G.$ The translation axis of $g$ is then given by $h A(o)$.
\begin{lem}\label{zzz} Let $G$ be a simple real algebraic group of rank one. For any loxodromic element $g\in G$ with translation axis $L$ and any sequence $x_i\in X$ such that  $d(x_i, L)\to \infty$,
we have $d(x_i, gx_i)\to \infty$.
\end{lem}

\begin{proof}
Without loss of generality, we may assume $g=m^{-1}a_{-s_0}\in MA$ with $s_0\ne 0$ so that $L=A(o)$.
Let $x_i\in X$ be a sequence such that $d(x_i, A(o))\to \infty$ as $i\to \infty$.
Write $x_i= n_ia_{-t_i}(o)$ with $n_i\in N$ and $t_i\in \br$.

We may then write
$$d(gx_i, x_i)= d(a_{t_i} h_i  n_i a_{-t_i}, a^{-1} o).$$
where $h_i= m a_{s_0}  n_i^{-1} a_{-s_0}m^{-1}\in N $.
As $d(x_i, A(o))\to \infty$, we have $a_{t_i}n_i a_{-t_i}\to \infty$.
It suffices to show
$a_{t_i} h_in_i a_{-t_i}\to \infty$.

 By the assumption that $G$ has rank one,
there is only one simple root, say $\alpha$, and
$\fn$ is the sum of at most two root subspaces $\fn=\fn_{\alpha} +\fn_{2\alpha}$ where $[\fn, \fn]= \fn_{2\alpha}$. Note that when $N$
is abelian, $\fn_{2\alpha}=\{0\}$.
Hence we have that for any $X, Y\in \fn$,
\be\label{log} \log\big( \exp( X)\exp( Y)\big)= X + Y +\frac{1}{2} [ X,  Y].\ee

Write $\log n_i=Y_i+Z_i$ with $Y_i\in \fn_{\alpha}$ and $Z_i\in \fn_{2\alpha}$. Since $\Ad_m $ preserves $\fn_{\alpha}$ and $\fn_{2\alpha}$,
we have 
$$\log h_i= -\Ad_{ ma_{s_0}}\log n_i=-e^{\alpha(s_0)} \Ad_m Y_i - e^{2\alpha (s_0)} 
\Ad_m Z_i.$$ Therefore by \eqref{log}, we get
$$\log h_i n_i= (1 -e^{\alpha(s_0)}\Ad_m ) Y_i +
(1-e^{2\alpha(s_0)}\Ad_m) Z_i  -\frac{1}{2} [e^{\alpha (s_0)}\Ad_m Y_i, Y_i].$$
Hence \begin{multline*}
    \Ad_{a_{t_i}} \log h_in_i
=  (1 -e^{\alpha(s_0)} \Ad_m    ) e^{\alpha(t_i)} Y_i  +\\
(1-e^{2\alpha(s_0)}\Ad_m) e^{2\alpha(t_i) }  Z_i  -  [e^{\alpha (s_0)} \Ad_m e^{\alpha(t_i)} Y_i, e^{\alpha(t_i)} Y_i].
\end{multline*}

Now suppose that $a_{t_i} h_in_i a_{-t_i}$ does not go to infinity as $i\to \infty$. By passing to a subsequence,
we may assume that $\Ad_{a_{t_i}}\log h_in_i$ is uniformly bounded. It follows that both sequences
$(1 -e^{\alpha(s_0)} \Ad_m    ) e^{\alpha(t_i)} Y_i$ and 
$ (1-e^{2\alpha(s_0)}\Ad_m) e^{2\alpha(t_i) }  Z_i  -  [e^{\alpha (s_0)} \Ad_m e^{\alpha(t_i)} Y_i, e^{\alpha(t_i)} Y_i]$ are uniformly bounded. Since $\alpha(s_0)\ne 0$, we have $e^{\alpha(t_i)} Y_i $ is uniformly bounded, which then implies that 
$e^{2\alpha(t_i)} Z_i$ is uniformly bounded. This implies that $\Ad_{a_{t_i}} \log n_i=e^{\alpha(t_i)}Y_i +e^{2\alpha(t_i)} Z_i$ is uniformly bounded, contradicting the hypothesis that
$d(a_{t_i}n_i a_{-t_i})\to \infty$ as $i\to \infty$.
This proves the claim.
\end{proof}

Let $\Ga<G$ be a discrete subgroup.
For $x=[g]\in \Gamma\ba G$, the injectivity radius $\op{inj} x$ is defined as
the supremum $r>0$ such that the ball $B_r(g)=\{h\in G: d(h,g)<r\}$
injects to $\Ga\ba G$ under the canonical quotient map $G\to \Ga\ba G$.
The injectivity radius of $\Ga\ba G$ is defined as $\op{inj}(\Ga\ba G)=\sup_{x\in \Ga\ba G} \op{inj} (x)$.

\begin{prop} \label{inj1} 
For any Anosov subgroup $\Ga<G$ which is not a cocompact lattice in a rank one group $G$, we have $\op{inj}(\Ga\ba G)=\infty$.
\end{prop}
\begin{proof} If $G$ has rank one, $\G$ is a convex cocompact subgroup which is not a cocompact lattice. 
In this case, take any $\xi\in \partial X$ which is not a limit point, and any $g_i\in G$ such that
$g_i(o)\to \xi$. Then $\op{inj}(g_i(o))\to \infty$ as $i\to \infty$.

Now suppose $\op{rank} G\ge 2$. We first observe that $\op{Vol}(\Ga\ba G)=\infty$; otherwise, $\G<G$ is a co-compact lattice, as Anosov subgroups consist only of loxodromic elements.
Since any Anosov subgroup $\Gamma$ is a Gromov hyperbolic group as an abstract group (\cite{KLP}, \cite{BPS}), it follows that $G$ is a Gromov hyperbolic space and hence must be of rank one, which contradicts the hypothesis.

If every simple factor of $G$ has rank at least $2$, the claim $\op{inj}(\Ga\ba G)=\infty$ follows from a more general result of  Fraczyk and Gelander \cite{FG} which applies to all discrete subgroups of infinite co-volume. Therefore it remains to consider the case where
$G=G_1\times G_2$ where $G_1$ and $G_2$ are respectively semisimple real algebraic subgroups of rank at least one and of rank precisely one. Let $\Sigma$ be a finitely generated group and $\pi:\Sigma\to G$ be an Anosov representation with $\Gamma=\pi(\Sigma)$ as in Definition \ref{Ano}. Let $\pi_i:\Sigma\to G_i$ be the composition of
$\pi$ and the projection $G\to G_i$ for each $i$. It follows from \eqref{dis} that
$\pi_i(\Sigma)$ is a discrete subgroup of $G_i$ for each $i=1,2$.
Let $X_i$ denote the rank one symmetric space associated to $G_i$ and let $X$ denote the Riemmanian product $X=X_1\times X_2$. Let $R>0$ be an arbitrary number. We will find a point $x\in X$ with $\op{inj}(x)\ge R$, i.e.,
 $d(x, \ga x)>R$ for all non-trivial $\ga \in \Ga$; this implies the claim. Choose any $x_1\in X_1$. By the discreteness of $\pi_1(\Sigma)$,
the set $\{\sigma\in \Sigma - \lbrace e\rbrace:
d_1(\pi_1(\sigma) x_1, x_1)<R\}$ is finite, which we write as $\{\sigma_1, \cdots, \sigma_m\}$.
For each $\sigma\in\Sigma\setminus\lbrace e\rbrace$, define a subset $T_2(\sigma)\subset X_2$ by 
$$T_2(\sigma)=\{z\in X_2: d_2(\pi_2(\sigma) z, z)<R\}.$$ Note that $\pi_2(\sigma)$ is a loxodromic element of $G_2$ and $T_2(\sigma)$ is contained in a bounded neighborhood of the translation axis of $\pi_2(\sigma)$ by Lemma \ref{zzz}. In particular,
the symmetric space $X_2$ is not covered by the finite union $\bigcup_{j=1}^m T_2(\sigma_j)$. Hence
we may choose $x_2\in X_2$ outside of $\bigcup_{j=1}^m T_2(\sigma_j)$.
We now claim that 
the injectivity radius at $x:=(x_1, x_2)$ is at least $R$; suppose not. Then
for some $\sigma\in \Sigma - \lbrace e\rbrace$, $d(
(\pi_1 (\sigma)x_1, \pi_2(\sigma)x_2), x)<R$. In particular, for $i=1,2$, $d_i (\pi_i(\sigma) x_i, x_i)<R$.
It follows that $\sigma=\sigma_j$ for some $1\le j\le m$ and
$x_2 \in T_2(\sigma_j)$, contradicting the choice of $x_2$. This proves the claim. \end{proof}

Theorem \ref{iinj} follows from Proposition \ref{inj1} and the following proposition, which was suggested by C. McMullen.
 \begin{prop}\label{inj0} Let $\Ga<G$ be a discrete subgroup with
$\op{inj}(\Ga\ba G)=\infty$.
Then
$$L^2(G)\propto L^2(\Ga\ba G)\quad\text{ and }\quad \sigma(X)\subset \sigma(\Ga\ba X).$$

\end{prop}

\begin{proof}
To prove the first claim, we need to show that the diagonal matrix coefficients of $L^2(G)$ can be approximated by the diagonal matrix coefficients of $L^2(\Ga\ba G)$ uniformly on compact subsets of $G$.

Let $v$ be any element of $L^2(G)$ and $\cal{K}\subset G$ a compact subset containing $e$. We will use the fact that $\mathrm{inj}(\Ga\ba G)=\infty$ to construct a sequence of functions $\lbrace F_i\rbrace\subset C_c(\Ga \ba G)$ such that
$$\lim_{i\rightarrow\infty} \max_{g\in\cal{K}} \left | \langle g.v,v\rangle_{L^2(G)}- \langle g. F_i, F_i\rangle_{L^2(\Ga\ba G)} \right|=0,$$
as required.
By the density of $C_c(G)$ in $L^2(G)$, there exists a sequence $\lbrace f_i\rbrace \subset C_c(G)$ such that $\lim_{i\rightarrow\infty}\| f_i-v\|_{L^2(G)}=0$, hence 
\begin{equation}\label{fapprox}
\lim_{i\rightarrow\infty} \max_{g\in\cal{K}} \left | \langle g.v,v\rangle_{L^2(G)}- \langle g. f_i, f_i\rangle_{L^2(G)} \right|=0.
\end{equation}
For each $i\ge 1$, we let $R_i>0$ be such that $(\text{supp} f_i) \cal{K} \subset B_{R_i}(e)$. Since $\mathrm{inj}(\Ga\ba G)=\infty$, there then exists a sequence $\lbrace g_i\rbrace\subset G$ such that $g_i B_{R_i}(e)$ injects to $\Ga\ba G$, i.e.\ the map $h\mapsto \Ga h$ is injective on $g_i B_{R_i}(e)$. 

For each $i$, consider the function $F_i\in C_c(\Ga\ba G)$ given by 
$$F_i(x)=\sum_{\gamma\in \Gamma}  f_i(g_i^{-1}\gamma h)\quad\text{ for any $x=[h]\in \Ga\ba G$}.$$

We then have that for any $g\in G $,
\begin{align*}
&\la g. F_i, F_i\ra_{L^2(\Ga\ba G)} = \int_{\Ga\ba G} F_i(xg) F_i(x)\, dx\\ &= \int_{\Ga\ba G} F_i(\Gamma h g) \left(\sum_{\ga\in \Ga} f_i(g_i^{-1}\ga h)\right)\, d(\Gamma h)
= \int_{G} F_i(\Gamma h g) f_i(g_i^{-1}h)\, dh 
\\ &= \int_{G} \left(\sum_{\ga\in \Ga }f_i(g_i^{-1}\ga g_i hg )\right)  f_i(h)  \,dh.
\end{align*}
We now observe that for $h\in G$, $g\in \mathcal{K}$, and $\gamma\in \Ga$,
$ f_i(g_i^{-1}\gamma g_i hg)f_i(h)\neq 0$ implies that $g_i^{-1}\gamma g_i hg \in B_{R_i}(e)$ and $h \in \mathrm{supp}\,f$. This in turn implies that both $\gamma g_i hg$ and $g_i h g$ are in $g_i B_{R_i}(e)$. Since $g_i B_{R_i}(e)$ injects to $\Gamma\ba G$, we must then have $\gamma= e$. Hence for all $i\ge 1$ and $g\in \mathcal{K}$,
\begin{align*} & \la g. F_i, F_i\ra_{L^2(\Ga\ba G)} =\int_G  \left(\sum_{\ga\in \Ga }f_i(g_i^{-1}\ga g_i hg )\right)  f_i(h) dh
\\& =\int_G f_i(g_i^{-1}e g_i hg )  f_i(h) dh =\int_G f_i(hg)f_i(h) dh = \langle g.f_i,f_i\rangle_{L^2( G)}.\end{align*}
 Combined with \eqref{fapprox}, this proves the first claim.

In order to prove the second claim, 
let $W^1(\Ga \ba X)\subset L^2(\Ga \ba X)$ be as defined in the proof of Theorem \ref{atmost}. Let $\lambda\in \sigma(X)$.
By Weyl's criterion (Theorem \ref{weyl}), there exists a sequence of $L^2(X)$-unit vectors  $\lbrace u_n\rbrace_{n\in\NN}\subset W^1(X)$ such that
$$ \lim_{n\rightarrow\infty} \|(\Delta +\lambda)u_n\|_{L^2(X)}=0.$$
Since $C_c^{\infty}(X)$ is dense in $W^1(X)$ with respect to $\|\cdot\|_{W^1(X)}$, we may assume that $\lbrace u_n\rbrace_{n\in\NN}\subset C_c^{\infty}(X)$. Since $\Ga\ba X$ has infinite injectivity radius, for each $n\in \NN$, we can find $g_n\in G$ so that $g_n  \text{supp}(u_n)$ injects to $\Ga\ba G$. We may therefore define $\lbrace v_n\rbrace_{n\in\NN}\subset W^1(\Ga\ba X)$ by
$$ v_n(\Gamma g_n x)=\begin{cases} u_n(x)\qquad & \mathrm{if}\, x\in  \text{supp}(u_n)\\ 0\qquad & \mathrm{otherwise.} \end{cases}$$ 
The $G$-invariance of $\Delta$ then gives 
$$  \lim_{n\rightarrow\infty} \|(\Delta +\lambda)v_n\|_{L^2(\Ga\ba X)}=\lim_{n\rightarrow\infty} \|(\Delta +\lambda)u_n\|_{L^2(X)}=0;$$
and so using Weyl's criterion again yields $\lambda\in\sigma(\Ga\ba X)$. Hence $\sigma(X)\subset \sigma(\Ga \ba X)$, as claimed.   
\end{proof}

\section{Temperedness of $L^2(\Ga\ba G)$}\label{sect}
Let $G$ be a connected semisimple real algebraic group and $\Ga<G$ be a Zariski dense discrete subgroup. The goal of this section is to prove Theorem \ref{tem} and Corollary \ref{tempc}.

\subsection*{ Burger-Roblin measures.}
 We set $N^+=w_0 N w_0^{-1}$ and $N^-=N$. For a $(\Ga, \psi)$-conformal measure $\nu_o$ on $\F$, or equivalently
 for a $(\Ga, \psi)$-conformal density $\nu=\{\nu_x:x\in X\}$, 
we denote by $m_\nu^{\BR}$ and $m_\nu^{\BR_*}$ the associated $N^+$ and $N^-$-invariant Burger-Roblin measures on $\Ga\ba G$ respectively, as defined in \cite{ELO}. By \cite[Lem. 4.9]{ELO}, it can also be defined as follows: for any $f\in C_c(\Ga\ba G)$,
$$m^{\BR}_\nu(f)=\int_{[k] m (\exp a) n\in K/M \times M AN^+} f([ k] m (\exp a ) n) e^{-\psi\circ \i (a)} \, d\nu_o(k^-) dm dadn $$
and 
$$m^{\BR_*}_\nu(f)=\int_{[k] m (\exp a) n\in K/M \times M AN^-} f([ k] m (\exp a ) n) e^{\psi (a)} \, 
d\nu_o(k^+) dm dadn $$
where $dm, da, dn$
are Haar measures on $M, \fa, N^{\pm}$ respectively.

Recall that $dx$ denotes the $G$-invariant measure on $\Gamma\ba G$
which is defined using the $(G, 2\rho)$-conformal measure, that is, the $K$-invariant probability measure on $\cal F$
(see \cite[(3.11)]{ELO}).
For real-valued functions $f_1, f_2$ on $\Ga\ba G$,
we write
$$\langle f_1, f_2\rangle=\int_{\Gamma\ba G} f_1(x) f_2(x) \, dx$$
whenever the integral converges. We write $C_c(\Gamma\ba G)_K$ for the space of $K$-invariant compactly supported continuous functions on $\Gamma\ba G$.
\begin{lem}\label{bbr} For a $(\Ga, \psi)$-conformal density $\nu$ and
any $f\in C_c(\Gamma\ba G)_K$, we have
 $$m^{\BR}_\nu (f)=\la f, E_{\nu}\ra= m^{\BR_*}_\nu(f).$$
\end{lem}
\begin{proof}
If $g=(\exp b) nk\in AN^+K$, then 
$$\beta_{e^-} (go, o)=\beta_{e^+} (\exp(-\i(b)), o) = \i (b).$$
Hence
\begin{align*}
    m^{\BR}_\nu (f)&=
    \int_{KAN^+}\int_K f(k\exp b n k_0) e^{-\psi \circ \i (b)} dk_0 d\nu_o(k^-) db dn\\
    &= \int_G \int_K f(kg) e^{-\psi (\beta_{e^-}(go, o))} d\nu_o(k^-) dg\\
    &=\int_G f(g)\int_K e^{-\psi (\beta_{k^-} (go,o))}d\nu_o(k^-) dg=\langle f, E_{\nu} \rangle
\end{align*}
If $g=(\exp b)nk\in ANK$, then
$\beta_{e^+}(go, o)=-b$ and using this, the second identity can be proved similarly.
\end{proof}

\subsection*{ Local matrix coefficients for Anosov subgroups.}
In the rest of this section, we assume that 
$$\text{$\G<G$ is a Zariski dense Anosov subgroup}.$$

\begin{lem}\label{ps}
For any $\psi\in D_\Ga$, 
there exists a unique unit vector $u\in \fa^+$ and $0<c\le 1$ such that $c\psi(u)=\psi_\Ga(u)$. Moreover 
$u\in \inte\L$.
\end{lem}
\begin{proof}
Since $\psi_\Ga$ is  strictly concave \cite[Propositions 4.6, 4.11]{PS}, there exists $0<c\le 1$ and unique $u\in \L$ such that $c\cdot \psi (u)=\psi_\Ga(u)$. Moreover there is no linear form tangent to $\psi_\Ga$ at $\partial \L$ \cite{PS}, and hence $u\in \inte\L$.
\end{proof}

 For each $v\in \inte\L$, there exists a unique linear form $\psi_v\in D_\Ga^\star$ such that $\psi_v(v)=\psi_\Ga(v)$
and a unique $(\Ga, \psi_v)$-conformal density supported on $\La$ \cite[Corollary 7.8 and Theorem 7.9]{ELO},
which we denote by $\nu_v$. Hence  \cite[Theorem 7.12]{ELO}, together with Lemma \ref{bbr}, implies (let $r=\text{rank} \,G$):
\begin{thm} \label{m1}
For any $v\in \op{int}\L$, there exists $\kappa_v>0$ such that
 for all $f_1, f_2\in C_c(\Ga\ba G)_K$ and any $w\in \ker \psi_v$,
\begin{multline*} \lim_{t\to +\infty} t^{(r-1)/2} e^{t(2\rho-\psi_v)(tv+\sqrt t w)}  \langle \exp (tv+\sqrt t w) f_1,  f_2\rangle 
\\ =\kappa_v e^{-I(w)} \cdot \langle f_1, E_{\nu_{ \i(v)}} \rangle\cdot  \langle f_2, E_{\nu_v} \rangle \end{multline*}
where $I(w)\in \br$ is given as in \cite[7.5]{ELO}.
Moreover, the left-hand side is uniformly bounded over all
$(t, w)\in (0, \infty)\times \ker \psi_v$ such that
$tv+\sqrt t w\in \fa^+$
\end{thm}

\begin{thm}\label{tem} 
\begin{enumerate} \item We have $L^2(\Gamma\ba G)$ is tempered if and only if $\psi_\Ga\leq\rho$.
\item If $L^2(\Ga\ba G)$ is tempered, then 
$$\lambda_0(\Ga\ba X)=\|\rho\|^2\quad\text{and}\quad  \sigma (\Ga\ba X)=[\|\rho^2, \infty).$$
\end{enumerate}
\end{thm}
\begin{proof} The second claim follows from Theorems \ref{tlam} and \ref{iinj}. Suppose that $\psi_\Ga\leq\rho$.
In order to show that $L^2(\Ga\ba G)$ is tempered, by Proposition \ref{chh},
it suffices to show that the matrix coefficients $g\mapsto \langle g.f_1,f_2\rangle$ are in $L^{2+\epsilon}(G)$ for all $\epsilon>0$ and for all $f_1,f_2\in C_c(\Ga\ba G)$, since $C_c(\Ga\ba G)$
 is dense in $L^2(\Ga\ba G)$.
Without loss of generality, we may just consider non-negative functions 
$f_1,f_2\in C_c(\Ga\ba G)$. Fix any $\e>0$.
Then using the Cartan decomposition $G=KA^+K$, we have 
\begin{align*}
\int_G \langle g. f_1, f_2\rangle^{2+\epsilon}\,dg=\int_K \int_{\fa^+}\int_K\langle k_1\exp(v)k_2. f_1, f_2\rangle^{2+\epsilon}\,\Xi(v)\,dk_1 \,dv\,dk_2,     
\end{align*}
where $\Xi(v)\asymp e^{2\rho(v)}$ (cf. \cite{Knapp1}).
Denoting $F_i(\Gamma g)=\max_{k\in K}  f_i(\Gamma g k) \in C_c(\Ga\ba G)_K$, we then have
$$\int_G \langle g.f_1, f_2\rangle^{2+\epsilon}\,dg \ll \int_{\fa^+} \langle \exp(v).F_1,F_2\rangle^{2+\epsilon} e^{2\rho(v)}\,dv. $$

Since $\psi_{\Ga}\leq\rho$, we have $\rho\in D_{\Ga}$. By Lemma \ref{ps},
 there exists $0<c\leq 1$ such that $c\rho\in D_{\Ga}^{\star}$ and a unit vector $u_0\in \inte \L$ such that 
 $$\psi_{\Ga}(u_0)=c\rho(u_0).$$ We now parameterize $\fa^+$ as follows: for each $v\in\ker\rho$, define
$$t_v:=\min\lbrace t\in\RR_{> 0}\,:\, tu_0+\sqrt{t}v\in \fa^+\rbrace.$$
Substituting $u=tu_0+\sqrt t v$ for $t\geq 0$ and $v\in \fb\cap\op{ker}\rho$
gives $du=s \cdot t^{\frac{r-1}{2}}\,dt\,dv$ for some constant $s>0$.
Then (letting $r=\dim(\fa)$)
\begin{align*}\int_{\fa^+}& \langle \exp(u).F_1,F_2\rangle^{2+\epsilon} e^{2\rho(u)}\,du
\\&\ll \int_{\ker\rho}\int_{t_v}^{\infty} \langle \exp(t u_0+\sqrt{t}v). F_1,F_2\rangle^{2+\epsilon} e^{2t\rho(u_0)}t^{(r-1)/2}\,dt\,dv. \end{align*}
By Theorem \ref{m1}, there exists $C=C(F_1,F_2)>0$ such that
$$ t^{(r-1)/2}  e^{(2-c)t\rho(u_0)} \langle  \exp(tu_0 +\sqrt t v) .F_1, F_2 \rangle \le C $$
for all $(v,t)\in \ker\rho\times [t_v,\infty)$. Combining this with the trivial bound $$ \langle g. F_1, F_2\rangle \leq \|F_1\|\|F_2\| ,$$ we have (again, for all $(v,t)\in \ker\rho\times [t_v,\infty)$),\begin{align*} &\langle \exp(t u_0+\sqrt{t}v). F_1,F_2\rangle^{2+\epsilon} \\ &\leq (C+\|F_1|\|F_2\|)^{2+\epsilon} \left(\min\left\lbrace 1, t^{-(r-1)/2} e^{-(2-c)t\rho(u_0)}\right\rbrace\right)^{2+\epsilon}\\ & \ll \min\lbrace 1, e^{-\eta t \rho(u_0)}\rbrace\leq e^{-\eta t \rho(u_0)} , \end{align*}where $\eta = (2-c) (2+\epsilon)>2$. This gives 
\begin{align*}\int_G \langle g. f_1, f_2\rangle^{2+\epsilon}\,dg \ll & \int_{v\in \ker\rho}\int_{t_v}^{\infty} e^{-\eta t \rho(u_0)} e^{2t\rho(u_0)}t^{(r-1)/2}\,dt\,dv\\&\ll \int_{\fa^+} e^{-(\eta-2)\rho(u)}\,du <\infty. \end{align*}
Therefore $L^2(\Ga\ba G)$ is tempered.

The converse holds for a general discrete subgroup.
Suppose now that $L^2(\Ga\ba G)$ is tempered.
Then by the definition of temperedness and the estimate
of $\Xi_G(g)$ in \eqref{hc}, it follows that
for any $\e>0$, there exists $C_\e>0$ such that
for any $f_1, f_2\in L^2(\Ga\ba G)_K$ and $u\in \fa^+$,
\be\label{cc} | \langle  \exp(u) . f_1, f_2 \rangle | \le 
C_\e \|f_1\| \|f_2\| e^{-(1-\e)\rho (u)}.
\ee 

Applying \cite[Prop. 7.3]{LO2}, we get $\psi_\Ga\le \rho$.
\end{proof}

Now recall the following recent theorem of Kim, Minsky, and Oh \cite{KMO}: 
\begin{thm} \cite{KMO}\label{kmo} Let $\G$ be an Anosov subgroup of the product $G$ of at least two simple real algebraic groups or $\G<G=\PSL_{d}(\br)$ be a Zariski dense Anosov subgroup
of a Hitchin subgroup.
Then $$\psi_\Ga \le \rho .$$
\end{thm}

Hence by Theorem \ref{tem}, we get:
\begin{cor}\label{tempc}  Let $\Ga<G$ be as in Theorem \ref{kmo}.
Then $L^2(\Ga\ba G)$ is tempered.
\end{cor}

\subsection*{Proofs of Theorem \ref{m33}}
The equivalence $(1)\Leftrightarrow (2)$ is proved in Theorem \ref{tem}.
The equivalence $(2)\Leftrightarrow (3)$ follows from Theorems \ref{iinj} and \ref{tem}.
When $\text{rank }G\ge 2$, (4) holds for any Anosov subgroup by 
Corollary \ref{ttt}. When $\text{rank }G=1$,
the implication $(1)+(2) \Rightarrow (4)$ is due to Sullivan \cite{Su} (see also \cite[Theorem 3.1]{RT})
when $X$ is a real hyperbolic space and to \cite[Theorem 1.1 and Proposition 5.1]{WW} in general.

\end{document}